\documentclass[10pt,letterpaper]{amsart}

\usepackage{amssymb}
\usepackage{bm}
\usepackage{xcolor}
\usepackage[pass]{geometry} % Use the class paper size and margins.

\numberwithin{equation}{section}
\newtheorem{Theorem}{Theorem}[section]
\newtheorem{Fact}[Theorem]{Fact}
\newtheorem{Corollary}[Theorem]{Corollary}
\newtheorem{Lemma}[Theorem]{Lemma}
\newtheorem{Proposition}[Theorem]{Proposition}
\theoremstyle{remark}
\newtheorem{Remark}[Theorem]{Remark}
\theoremstyle{definition}
\newtheorem{Definition}[Theorem]{Definition}
\newtheorem{Example}[Theorem]{Example}
\newtheorem*{acknowledgements}{Acknowledgements}

\newcommand{\R}{\mathbb{R}}
\newcommand{\C}{\mathbb{C}}
\newcommand{\Q}{\mathbb{Q}}
\newcommand{\Z}{\mathbb{Z}}
\newcommand{\K}{\mathbb{K}}
\newcommand{\mc}{\mathcal}
\newcommand{\mb}{\mathbf}
\newcommand{\op}{\operatorname}
\renewcommand{\phi}{\varphi}
\renewcommand{\epsilon}{\varepsilon}

\title[Images of wave-front singularities]{
Explicit analytic functions defining the images
of simple wave-front singularities
}
\author{K.~Saji}
\address[Kentaro Saji]{
Department of Mathematics,
  Faculty of Science,
  Kobe University,
  Rokko, Kobe 657-8501}
\email{saji@math.kobe-u.ac.jp}

\author{M.~Umehara}
\address[Masaaki Umehara]{
Department of Mathematical and Computing Sciences,
Institute of Science Tokyo,
2-12-1-W8-34, O-okayama, Meguro-ku,
Tokyo 152-8552, Japan.
}
\email{umehara@comp.isct.ac.jp}

\author{K.~Yamada}
\address[Kotaro Yamada]{
National Institute of Technology\\
   Numazu College\\
   3600 Ooka, Numazu, Shizuoka 410-8501,
   Japan
}
\email{kotaro@numazu-ct.ac.jp}

\dedicatory{Dedicated to Professor Wayne Rossman on his sixtieth birthday}
\date{November 28, 2026}

\thanks{The first author was supported in part
by Grant-in-Aid for Scientific Research (C)
No. 25K07001. The second and the third authors were
supported in part
by Grant-in-Aid for Scientific Research
(B) No. 23K20794 and (B) No. 23K22392 respectively,
from Japan Society for the Promotion of Science.}

\begin{document}
\setcounter{tocdepth}{1}

\begin{abstract}
We construct explicit main-analytic functions for the images of the standard
maps associated with simple wave-front singularities of types $A$, $D$, and $E$.
The construction uses resultants of explicitly specified pairs of
one-variable polynomials whose coefficients are expressed directly in
the target coordinates of the standard maps.
For types $A$ and $D$, it recovers one-variable discriminant representations.
Using this, we give direct proofs of sharp bounds
on the number of preimages of the standard map for types $A$ and $D$.

For each of the exceptional types $E_6$, $E_7$,
and $E_8$, there exists an additional square
factor in the resultant.
Removing this factor gives the corresponding main-analytic function.

\end{abstract}

\maketitle

\tableofcontents

\section*{Introduction}
We consider the images of the standard maps defining wave-front
singularities of types $A$, $D$, and $E$.
Our aim is to describe these images, up to subsets of smaller Hausdorff
dimension, as zero sets of single real analytic functions, called
``main-analytic functions.''

Main-analyticity, introduced in a variant form by
Ishikawa--Koike--Shiota~\cite{IKS}, will be used here for images of
real analytic maps which are globally controlled by one real analytic
function.
A subset $S\subset\R^n$ is called a \emph{global main-analytic set} if there
exists a real analytic function $\Theta$ on $\R^n$ whose zero set
$\mathcal Z(\Theta)$ contains~$S$, has the same Hausdorff dimension as~$S$,
and differs from~$S$ only in a subset of strictly smaller Hausdorff dimension.
Such a $\Theta$ is called a \emph{main-analytic function} for~$S$.

We begin with a simple example:
The real analytic map
$
h:\R^2\ni (u,v)\mapsto (u,uv,v^2)\in \R^3
$
gives the standard map of a {\it cross cap}.
If we set $h(u,v)=(x,y,z)$,
then we have
$$
u=x,\quad uv=y,\quad v^2=z,
$$
and
a main-analytic function for $h(\R^2)$ is given by
$\Theta(x,y,z):=y^2-zx^2$,
which coincides with the resultant between
$A(v):=xv-y$ and $B(v):=v^2-z$.
In fact, the first principal subresultant coefficient (cf.\ Appendix~A)
of $A(v)$ and $B(v)$ is given by
$
\mc S:=\op{Psc}_v(A,B)=x
$,
and the difference
$\mc Z(\Theta)\setminus h(\R^2)$
is contained in the
zero set
\begin{equation}\label{eq:185}
\{x=y=0\}=\mc Z(\Theta)\cap
\mc Z(\mc S).
\end{equation}

\medskip
In Ishikawa--Koike--Shiota~\cite{IKS}, a criterion ensuring the existence
of main-analytic functions for the critical value sets of
proper real analytic maps is given. 
Appendix~E also provides
a criterion for the main-analyticity of images of
polynomial maps. Thus, many images of maps with only
generic singular points are expected to be global main-analytic sets.

Our interest in main-analyticity arose from the problem of
extending the image of a real analytic map.
In a paper in preparation~\cite{FKKRUYY}, we consider a real analytic
map $f\colon X^m\to \R^n$ on a real analytic manifold, where $n\ge m$.
Suppose that each point of $X^m$ has a neighborhood $U$ such that
$f(U)$ is a main-analytic set.
Under a weak properness condition on $f$, we show that $f(X^m)$ admits
no nontrivial $m$-dimensional real analytic extension.

Simple singularities of types $A$, $D$, and $E$
(collectively referred to as ADE type)
have long been studied through their deformations
and discriminants.
Arnol'd~\cite{A} classified the simple critical points
of analytic functions.
Their discriminants are closely related to the orbit
spaces of the corresponding Weyl groups.
Saito~\cite{Sa1} developed flat structures on these
orbit spaces.
Looijenga~\cite{Loo} studied real discriminants
using the period map.

\medskip
Let
\[
S=f(\R^{k-1})\subset\R^k
\]
be the image of a standard polynomial map of ADE type.
The image of its complexification is the complex
discriminant of the corresponding generating family.
In particular, there exists  a polynomial
$\delta_S\in\R[x_0,\ldots,x_{k-1}]$
with no repeated irreducible factors such that
\[
f^\C(\C^{k-1})
=\mc Z_\C(\delta_S)
=\{\mb x\in\C^k\,;\,\delta_S(\mb x)=0\}.
\]
The polynomial $\delta_S$ is determined up to a nonzero real constant factor.
We refer to it as the {\it discriminant polynomial
of the generating family}.

The real zero set of $\delta_S$ contains $S$,
but may also contain additional points.
Looijenga's result~\cite[Proposition~(1.1)]{Loo}
implies that, near the origin, these additional points
form a set of real dimension at most $k-2$.
Since each standard map of ADE type is weighted homogeneous
with positive weights, the local dimension bound holds globally.
Together with $\dim_H S=k-1$, this shows that
$\delta_S$ is a main-analytic function for $S$.
Our purpose is to give explicit formulas for $\delta_S$
in terms of the resultant of two polynomials
and to extract further information about the real images from
the pair of polynomials. We first show the following:

\medskip
\noindent
{\bf Theorem~A.}
{\it
Let $h:\mathbb{R}^{k-1}\to\mathbb{R}^k$ be a standard map defining
a wave-front singularity of type $A$, $D$, or $E$.
Let $x_0,\ldots,x_{k-1}$ be the standard coordinates on the target
space $\mathbb R^k$, and set
\begin{equation}\label{eq:X}
\mathbf{x}:=(x_0,\ldots,x_{k-1}).
\end{equation}
There exist explicit polynomials
$A^{\mathbf{x}}(v),\,B^{\mathbf{x}}(v)\in \mathbb{R}[\mathbf{x}][v]$
and $r(\mathbf{x})\in \mathbb{R}[\mathbf{x}]$
such that the following hold.
\begin{itemize}
\item
 If $h$ is of type $A$ or $D$, then $r(\mathbf{x})=1$. If $h$ is of type $E$, then $r(\mathbf{x})$ is irreducible over $\mathbb{C}$.
\item $r(\mathbf{x})^2$ divides 
$\operatorname{Res}_v\!\bigl(A^{\mathbf{x}},B^{\mathbf{x}}\bigr)$.
\item A main-analytic function $\Theta$ for $h(\mathbb{R}^{k-1})\subset\mathbb{R}^k$ is given by
\begin{equation}\label{eq:Theta0}
  \Theta(\mathbf{x})
  = \frac{\operatorname{Res}_v\!\bigl(A^{\mathbf{x}},
      B^{\mathbf{x}}\bigr)}{r(\mathbf{x})^2} \in \mathbb{R}[\mathbf{x}].
\end{equation}
\end{itemize}
Moreover, the polynomial $\Theta$ is irreducible over $\mathbb{C}$
and coincides with the discriminant polynomial
$\delta_S$ up to a nonzero scalar.
}

\medskip
For type $A$, this construction gives the classical discriminant of a
polynomial in one variable. For type $D$, it recovers the one-variable
discriminant representation of
Holweck--Luque--Planat~\cite[Lemma~1]{HLP}.
Dividing the resultant by $r^2$ gives the discriminant
polynomial $\delta_S$ up to a nonzero constant factor.

For types $A$ and $D$, \eqref{eq:Theta0}
also gives direct proofs of the following sharp bounds on the number of real preimages.
The bounds for type $A$ are well-known, and those for type $D$ can also
be read off from Vassiliev's work~\cite{Vassiliev2024}.

\medskip
\noindent
{\bf Corollary~B.}
{\it
If $h$ is of type $A$ or $D$, then
for any $\epsilon>0$,
the following estimates hold
for the maximum number
$
M:=\max_{|\mathbf{x}|<\epsilon}\# h^{-1}(\mathbf{x})
$
of preimages.
\begin{enumerate}
\item If $h$ is of type $A_{2m}$ $($resp.\ $A_{2m+1})$, then
$M=m$ $($resp.\ $m+1)$.
\item If $h$ is of type $D_{2m+1}$, then $M=m+1$.
\item If $h$ is of type $D_{2m}^-$, then $M=m+1$.
\item If $h$ is of type $D_{2m}^+$, then $M=m$.
\end{enumerate}
}

\medskip
Let
\begin{equation}\label{eq:313}
s_0(\mathbf{x})
:=
\operatorname{Psc}_v\!\bigl(A^{\mathbf{x}}(v),B^{\mathbf{x}}(v)\bigr)
\end{equation}
denote the first principal subresultant coefficient
(cf.\ Appendix~A).
For types $A$ and $E$, we set
$$
s(\mathbf{x}) := s_0(\mathbf{x}).
$$
For type $D$, we set
\begin{equation}\label{342}
s(\mathbf{x})
:=
(x_0^2+x_1^2)s_0(\mathbf{x}).
\end{equation}
The following theorem describes the image of the self-intersection locus
and the set $\mc Z(\Theta)\setminus h(\R^{k-1})$ for types $A$, $D$,
and $E$.

\medskip
\noindent
{\bf Theorem~C.}
{\it
Let $h:\mathbb R^{k-1}\to\mathbb R^k$ be the standard map
of type $A$, $D$, or $E$, and let $\Theta$ be the
main-analytic function constructed in Theorem~A.
Then the following assertions hold.
\begin{enumerate}
\item
The image of the self-intersection locus
\[
\mathcal I_h:=\{p\in\R^{k-1}\,;\,
  \exists q\in\R^{k-1}\setminus\{p\},\ h(q)=h(p)\}
\]
satisfies
\begin{equation}\label{eq:349}
  h(\mathcal I_h)\subset \mathcal Z(s).
\end{equation}
If $h$ is of type $D$, then
$
  h(\mathcal I_h)\not\subset\mathcal Z(s_0).
$
Moreover,
$
  h^{-1}\bigl(\mathbb R^k\setminus \mathcal Z(s)\bigr)
$
is an open dense subset of $\mathbb R^{k-1}$.

\item
$
  \dim_H\bigl(\mathcal Z(s)\cap\mathcal Z(\Theta)\bigr)<k-1.
$

\item
For types $A$, $D$, and $E$, we have
$
  \mathcal Z(\Theta)\setminus h(\mathbb R^{k-1})
  \subset
  \mathcal Z(s).
$
Moreover, if $h$ is of type $D$, then
$
  \mathcal Z(\Theta)\setminus h(\mathbb R^{k-1})
  \not\subset \mathcal Z(s_0).
$
\end{enumerate}
}

\medskip
For types $A$, $D$, and $E$, (3) of Theorem~C  gives
\begin{equation}\label{eq:379}
\mc Z(\Theta)\setminus\mc Z(s)
=h(\R^{k-1})\setminus\mc Z(s).
\end{equation}
The non-inclusions in {\rm(1)} and {\rm(3)} show that, for type $D$,
the factor $x_0^2+x_1^2$ is needed both to capture the self-intersection
image and to contain the difference between the discriminant zero set
and the real image.

The paper is organized as follows.
Section~1 recalls the standard maps and the necessary preliminaries.
Sections~2 and~3 treat types $A$ and $D$, and Sections~4--6 treat
$E_6$, $E_7$, and $E_8$.
Appendices~A and~B recall properties of resultants and subresultants
and give dimension estimates for common zero sets.
Appendix~C presents the elimination argument shared by the exceptional
types, and 
Appendix~\ref{app:exceptional} proves that,
for each $\ell=6,7,8$, 
$r_\ell^2$ divides $\mc R_\ell$ and completes a self-contained proof that
$\Theta_{E_\ell}$ is irreducible over $\C$.
Appendix~E gives a general criterion for the global
main-analyticity of images of polynomial maps.

\section{Preliminaries}

Consider a $k$-dimensional proper real analytic map
$
h_S:(\R^k,o) \to (\R^l,\mb 0),
$
where $o$ is the origin of $\R^k$
and $\mb 0$ is the origin of $\R^l$.
We think of $h_S$ as the ``standard map'' of a singular point
$o$.

\begin{Definition}\label{Def:S}
Let $(k,l,m,n)$ be a quadruple of non-negative integers
satisfying $k\le m$, $l\le n$ and $m-k\le n-l$.
We consider the map
$$
\tilde h_S:\R^k\times \R^{m-k}\ni
(\mb x,\mb y)\mapsto (h_{\rm S}(\mb x),\mb y,\mb 0)
\in \R^l\times \R^{m-k}\times \R^{n-l-m+k}=\R^n,
$$
which is called the {\it suspension} of the map $h_S$.
\end{Definition}

\begin{Proposition}\label{prop:codim}
If $h_S(\R^k)$ is a global main-analytic set,
then so is $\tilde h_S(\R^m)$.
\end{Proposition}

\begin{proof}
Write a point of the target of the suspension as
$(\mb X,\mb Y,\mb Z)\in
\R^l\times\R^{m-k}\times\R^{n-l-m+k}$.
If $\Theta_{h_S}$ is a main-analytic function for $h_S(\R^k)$, then
$$
\widetilde\Theta(\mb X,\mb Y,\mb Z)
:=\Theta_{h_S}(\mb X)^2+|\mb Z|^2
$$
is a main-analytic function for $\tilde h_S(\R^m)$,
where, for a non-negative integer $s$,
$$
|\mb z|:=\sqrt{\sum_{i=1}^s|z_i|^2},
\qquad \mb z:=(z_1,\ldots,z_s)\in \R^s,
$$
with the empty sum interpreted as $0$ when $s=0$.
\end{proof}

\begin{Definition}\label{def:md}
Let $X^m$ be a real analytic $m$-manifold.
A real analytic map
$
 f:X^m\to\R^n\,\, (m\le n)
$
is called an \emph{$m$-dimensional real analytic map} if $\mathrm{rank}\,df=m$
on a dense open subset of $X^m$.
\end{Definition}

We shall use the following fact.

\begin{Proposition}\label{prop:HDIM}
If $f : X^m \to \mathbb{R}^n$ is an $m$-dimensional real analytic map,
then $f(X^m)$ has Hausdorff dimension $m$, that is,
$\dim_H f(X^m)=m$.
\end{Proposition}

\begin{proof}
At a point where $\operatorname{rank}df=m$,
there exists a neighborhood on which $f$ is an embedding.  Hence
$\dim_H f(X^m)\ge m$.  Conversely, a countable coordinate cover of $X^m$
and local Lipschitz continuity give $\dim_H f(X^m)\le m$.
\end{proof}

By Proposition~\ref{prop:codim},
it suffices to choose a standard map of
minimal possible codimension.
Here, we are concerned with wave-front singularities
that appear on hypersurfaces.
\begin{enumerate}
\item
Fix an integer $k(\ge 2)$.
The standard map $h_{A_k}:\R^{k-1}\to \R^k$
of an {\it $A_k$-singular point}
is defined by $($cf. \eqref{eq:X}$)$
\begin{equation}\label{eq:pA408}
h_{A_k}(v,\mb x_2):=
\Big(
k v^{k+1}+\sum_{i=2}^{k-1}(i-1)x_i v^i,
-(k+1)v^k-\sum_{j=2}^{k-1}j x_j v^{j-1},
 \mb x_2\Big),
\end{equation}
where $\mb x_2:=(x_2,\ldots,x_{k-1})$.
This map gives a singularity of corank one at the origin $o\in \R^{k-1}$.
\item
We fix an integer $k$ satisfying $k\ge 4$.
The standard map $h^{\pm}_{D_k}:\R^{k-1}\to \R^k$
of a {\it $D^{\pm}_k$-singular point}
is defined by
$$
h^{\pm}_{D_k}(u,v,\mb x_3):=
\Big(
h_0(u,v, \mb x_3),
h_1(u,v,\mb x_3),
h_2(u,v, \mb x_3),\mb x_3
\Big)
$$
and
\begin{align}\label{eq:374a}
h_0(u,v,\mb x_3)&:=
\pm 2u^2 v+(k-2)v^{k-1}+\sum_{i=2}^{k-1}(i-2)x_i v^{i-1}, \\
\label{eq:374b}
h_1(u,v,\mb x_3)&:=\mp 2uv, \\
\label{eq:374c}
h_2(u,v,\mb x_3)&:=
\mp u^2-(k-1)v^{k-2}-\sum_{i=3}^{k-1}(i-1)x_i v^{i-2},
\end{align}
where $\mb x_3:=(x_3,\ldots,x_{k-1})$.
The map $h:=h^{\pm}_{D_k}$
gives a singularity of corank two at the origin $o\in \R^{k-1}$
and has a $\pm$-ambiguity $h=h_+$ or $h=h_-$
$($see Remark~\ref{rmk:433}$)$.

\item
The standard map $h_{E_6}:\R^5\to \R^6$
of an {\it $E_6$-singular point}
is defined by
\begin{align}
h_0&:=2u^3+3v^4+x_3v^2+uv\,d^{\mb x}_6(v), \label{eq:1022a}\\
h_1&:=-3u^2-x_4v-x_5v^2, \label{eq:1022b}\\
h_2&:=-4v^3-2x_3v-ud^{\mb x}_6(v), \label{eq:1022c}\qquad
d^{\mb x}_6(v):=x_4+2x_5v.
\end{align}
This map gives a singularity of corank two at the origin $o\in \R^5$.
Similarly,
the standard maps $h_{E_7}:\R^6\to \R^7$,
and $h_{E_8}:\R^7\to \R^8$ of $E_7$ and $E_8$ singular points
are defined in Sections~5 and~6, respectively.
\end{enumerate}

The key to our proofs of Theorems~A and C
lies in the use of ``resultants'' of two polynomials.
We use standard properties of resultants and subresultants,
recalled in Appendix~A.
For two polynomials $a(v)$ and $b(v)$ whose
leading coefficients
are non-zero,
 $a(v)$ and $b(v)$
have a common factor of degree at least two in $\C$
if and only if
\begin{equation}
\op{Res}_v(a,b)=\op{Psc}_v(a,b)=0.
\end{equation}
Here the first principal subresultant coefficient
$\op{Psc}_v(a,b)$ is defined in Appendix~A.
Precise properties of subresultants
can be found in Gonz\'alez-Vega \cite{V}.

\section{Singularities of type $A$}

Let
$
h:=h_{A_k}:\R^{k-1}\to \R^k
$
$(k\ge2)$
be the standard map of an $A_k$-singular point
given by \eqref{eq:pA408}.
For a unified treatment, set $\K=\R$ or $\C$, and define
\begin{equation}\label{eq:hatH}
\hat h :=
\begin{cases}
h & \text{if } \K=\R,\\
h^\C & \text{if } \K=\C.
\end{cases}
\end{equation}
Set
\begin{equation}\label{eq:A_poly}
A^{\mb x}(v)
:=
v^{k+1}
+x_{k-1}v^{k-1}
+\cdots
+x_1 v
+x_0,
\end{equation}
where $\mb x=(x_0,\ldots,x_{k-1})\in \K^k$.

\begin{Proposition}\label{prop:A_image}
The image of $\hat h$ coincides with
$$
\mc W_{A_k}(\K):=
\bigl\{
\mb x\in\K^k
\, ;\,
\text{$A^{\mb x}$ has a root of multiplicity at least two in $\K$}
\bigr\}.
$$
\end{Proposition}

\begin{proof}
Let
$
B^{\mb x}(v)
:=
dA^{\mb x}/dv.
$
Then $A^{\mb x}$ has a root of multiplicity at least two in $\K$
if and only if there exists $v\in\K$ such that
$
A^{\mb x}(v)=B^{\mb x}(v)=0.
$
Solving $B^{\mb x}(v)=0$ for $x_1$
and substituting into $A^{\mb x}(v)=0$
gives \eqref{eq:pA408} of $\hat h$.
\end{proof}

\begin{Proposition}\label{prop:A_basic}
The map $h$ and its complexification
$h^\C:\C^{k-1}\to\C^k$ are proper.
Moreover,
for each $\mb x\in\K^k$ $(\K=\R,\C)$,
the inverse image $\hat h^{-1}(\mb x)$
consists of at most $k$ points,
and $\hat h^{-1}(\mb 0)=\{o\}$.
Furthermore,
$h$ is a $(k-1)$-dimensional real analytic map.
\end{Proposition}

\begin{proof}
Writing
$
h(v,\mb x_2)
=
(X_0(v,\mb x_2),
 X_1(v,\mb x_2),
 \mb x_2),
$
($\mb x_2:=(x_2,\ldots,x_{k-1})$)
the components $X_0,X_1$
are polynomials in $v$
with nonzero constant leading coefficients.
Hence, by Remark~\ref{rmk:Cauchy} below,
$h$ and $h^\C$ are proper.

The polynomial $dA^{\mb x}/dv$ has degree $k$, so the number
of possible values of $v$ is at most $k$.

The remaining assertions are immediate.
\end{proof}

\begin{Remark}[Cauchy's bound for roots]\label{rmk:Cauchy}
Let
$$
  f(t)=t^n+a_{n-1}t^{n-1}+\cdots+a_0
$$
be a polynomial with complex coefficients.  Then every root $\alpha$
of $f$ satisfies
$$
  |\alpha|
  \le
  1+\max_{0\le i\le n-1}|a_i|.
$$
In particular, if the coefficients vary within a bounded set, then
all roots remain uniformly bounded.
\end{Remark}

\begin{Theorem}\label{thm:A_main}
The image
$
h(\R^{k-1})
$
is a global main-analytic set in $\R^k$.
A main-analytic function is
\begin{equation}\label{eq:Ba}
\Theta_{A_k}
:=
\op{Res}_v
\bigl(
A^{\mb x},B^{\mb x}
\bigr)
\qquad
  \left(B^{\mb x}(v):=\frac{dA^{\mb x}}{dv}\right).
\end{equation}
\end{Theorem}

\begin{proof}
Let
$$
\mc L
:=
\Bigl\{
\mb x\in\mc Z(\Theta_{A_k})
\,;\,
\text{every common root of $A^{\mb x}$ and $B^{\mb x}$
is non-real}
\Bigr\}.
$$
By Proposition \ref{prop:A_image}, we have
$
h(\R^{k-1})
=
\mc Z(\Theta_{A_k})\setminus\mc L.
$
Since
$\dim_H h(\R^{k-1})=k-1$,
it suffices to show
$\dim_H\mc L<k-1$.
If $k=2$, then
$\mc L=\varnothing$.
So we may set $k\geq3$.  If $\mb x\in\mc L$, then
$A^{\mb x}$ has a non-real multiple root.  Choose such a root
$\alpha=a+\sqrt{-1}b$, where $a,b\in\R$ and $b\ne 0$.  Since
$A^{\mb x}$ has real coefficients, $\bar\alpha$ is also a multiple
root.  Hence
\begin{equation}\label{eq:713}
A^{\mb x}(v)
=
\bigl(v^2-2av+a^2+b^2\bigr)^2Q(v).
\end{equation}
First, suppose that $k=3$.  Then $Q(v)=1$.  Since the
coefficient of $v^3$ in $A^{\mb x}(v)$ vanishes,
the factorization \eqref{eq:713} gives $a=0$, and $\mb x$ is determined
by the single real parameter $b$.  Hence
\begin{equation}\label{eq:730}
\dim_H\mc L\leq1=k-2.
\end{equation}
Next, suppose that $k\geq4$.  Write
$$
Q(v)
=
v^{k-3}
+c_{k-4}v^{k-4}
+\cdots+c_0.
$$
The coefficient of $v^k$ in the right hand side of
\eqref{eq:713} is $c_{k-4}-4a$.  Since the coefficient of $v^k$
in $A^{\mb x}(v)$ vanishes, we have
$
c_{k-4}=4a.
$
Consequently, $\mb x$ is determined by the real parameters
$$
a,\ b,\ c_{k-5},\ldots,c_0,
$$
where the list $c_{k-5},\ldots,c_0$ is empty when $k=4$.
The number of these parameters is
$
2+(k-4)=k-2.
$
Consequently, $\mb x$ depends on at most $k-2$ real parameters.
Hence
$
\dim_H\mc L\leq k-2.
$
\end{proof}

\medskip
\begin{proof}[Proof of Theorems A and C for type $A$]
The assertions of Theorem~A except for the final one
follow from Theorem \ref{thm:A_main}.
In the complex setting, it is well-known that
$
h^\C(\C^{k-1})
=
\mc Z_\C(\Theta_{A_k}),
$
and $\Theta_{A_k}$
is irreducible over $\C$.
Thus $\Theta_{A_k}$
coincides with the discriminant polynomial
of type $A_k$. So we obtain the final assertion of Theorem~A.

We prove Theorem~C for type $A$.  Set (cf. \eqref{eq:Ba})
$$
  s_0(\mb x):=
  \operatorname{Psc}_v\bigl(A^{\mb x},B^{\mb x}\bigr).
$$
For type $A$, we have $s=s_0$ (cf. \eqref{eq:313}).
We suppose that $\mb x=h(v_1,\mb x_2)=h(v_2,\mb x_2)$ holds for distinct $v_1,v_2\in \R$.
By the
explicit computation in the proof of Proposition~\ref{prop:A_image},
each $v_i$ is a common root of $A^{\mb x}$ and $B^{\mb x}$.  Hence
the greatest common divisor
$\gcd(A^{\mb x},B^{\mb x})$ of the two polynomials
has degree at least two, so
$s_0(\mb x)=0$.  This proves \eqref{eq:349}.

Next, suppose that
\begin{equation}\label{eq:818}
\Theta_{A_k}(\mb x)=s_0(\mb x)=0.
\end{equation}
Let $\lambda_1,\ldots,\lambda_r\in\C$ be the distinct roots of
$A^{\mb x}$, with multiplicities $m_1,\ldots,m_r$, respectively.
By \eqref{eq:818},
$$
2
\le
 \deg_v\gcd(A^{\mb x},B^{\mb x})
=
\sum_{i=1}^r(m_i-1)
=
k+1-r,
$$
and hence $r\le k-1$.  Moreover, since the coefficient of $v^k$
in $A^{\mb x}$ vanishes,
\begin{equation}\label{eq:782}
\sum_{i=1}^r m_i\lambda_i=0.
\end{equation}
If $\lambda_i\in \C\setminus \R$,
its conjugate $\overline{\lambda_i}$
also belongs to
$\{\lambda_1,\ldots,\lambda_r\}$.
Thus,
by \eqref{eq:782},
the space of such
root configurations has real dimension at most $r-1(\le k-2 )$.
Since $A^{\mb x}$ is monic, its roots (with multiplicities) determine
$\mb x$.  Therefore we obtain
\begin{equation}\label{eq:830}
  \dim_H\bigl(\mathcal Z(\Theta_{A_k})\cap\mathcal Z(s_0)\bigr)
  \le k-2<k-1.
\end{equation}
This proves (2) of Theorem~C for type $A$.

To complete (1), note that the set
$h^{-1}(\R^k\setminus\mathcal Z(s_0))$ is open and that $h$ is injective
on it by \eqref{eq:349}.  If this set is not dense, the real analytic
function $s_0\circ h$ vanishes on a non-empty open set and hence
identically.  Since $h(\R^{k-1})\subset\mathcal Z(\Theta_{A_k})$, this
implies, by \eqref{eq:830},
\[
 k-1=\dim_H h(\R^{k-1})\\
 \le\dim_H\bigl(\mathcal Z(\Theta_{A_k})\cap\mathcal Z(s_0)\bigr)
 \le k-2,
\]
a contradiction.  Thus $h^{-1}(\R^k\setminus\mathcal Z(s_0))$ is open
and dense, proving the last statement of (1).

Finally, we set
$$
  \Delta_A:=\mathcal Z(\Theta_{A_k})\setminus h_{A_k}(\R^{k-1}).
$$
If $\mb x\in\Delta_A$, then $A^{\mb x}$ has a multiple root in $\C\setminus \R$
because of Proposition~\ref{prop:A_image}.
Thus, for every multiple root of $A^{\mb x}$,
its conjugate is also  a multiple root.
Hence $A^{\mb x}$ and $B^{\mb x}$ have at least two common roots, so
$s_0(\mb x)=0$.  Therefore
$$
  \Delta_A\subset\mathcal Z(s_0)=\mathcal Z(s),
$$
which proves (3) of Theorem~C for type $A$.
\end{proof}

\begin{proof}[Proof of Corollary B for type A]
Suppose that $k=2m$ (resp. $k=2m+1$), where $m$ is a positive integer.
Then the degree of $A^{\mb x}(v)$ is $2m+1$
(resp. $2m+2$).  Hence $A^{\mb x}$ has at most $m$
(resp. $m+1$) distinct real multiple roots, and so
$M\le m$ (resp. $M\le m+1$).
These bounds are attained arbitrarily close to the origin.  If $k=2m$,
choose sufficiently small pairwise distinct real numbers
$\lambda_1,\ldots,\lambda_m$ and set
$\mu:=-2\sum_{i=1}^m\lambda_i$, with $\mu\ne\lambda_i$ for all $i$.
The polynomial
\[
 A^{\mb x}(v)=(v-\mu)\prod_{i=1}^m(v-\lambda_i)^2
\]
has vanishing $v^{2m}$-coefficient. Its distinct multiple roots
are precisely $\lambda_1,\ldots,\lambda_m$, so
Proposition~\ref{prop:A_image} shows that $h^{-1}(\mb x)$ consists
of exactly $m$ points.  If $k=2m+1$, choose sufficiently small pairwise distinct real
numbers $\lambda_1,\ldots,\lambda_{m+1}$ with
$\sum_i\lambda_i=0$ and set
\[
 A^{\mb x}(v)=\prod_{i=1}^{m+1}(v-\lambda_i)^2.
\]
Then $h^{-1}(\mb x)$ consists of exactly $m+1$ points.  In both cases,
$|\mb x|$ can be made arbitrarily small.
\end{proof}

\begin{Example}[Cusp, $k=2$]
Here
$
h(v)=(2v^3,-3v^2).
$
We have
$$
A^{\mb x}(v)=v^3+x_1 v+x_0,
\qquad
\Theta_{A_2}=27x_0^2+4x_1^3.
$$
In this case
$
\mc Z(\Theta_{A_2})=h(\R).
$
\end{Example}

\begin{Example}[Swallowtail, $k=3$]
Here
$
A^{\mb x}(v)
=
v^4+x_2v^2+x_1v+x_0,
$
and
$\Theta_{A_3}
=
\op{Res}_v(A^{\mb x},A_v^{\mb x})$
is the usual quartic discriminant.
In this case
\[
\mc Z(\Theta_{A_3})\setminus h(\R^2)
=\{(t^2,0,2t):t>0\},
\]
which is a curve diffeomorphic to an open half-line.
\end{Example}

\section{Singularities of type $D$}

Let
$
  h:=h^\pm_{D_k}:\R^{k-1}\to\R^k
$
$(k\ge4)$
be the standard map given by
\eqref{eq:374a}, \eqref{eq:374b}, and \eqref{eq:374c}.
Set $\K=\R$ or $\C$, and define $\hat h=h$ or $h^\C$
as in \eqref{eq:hatH},
and set
$$
F(u,v,\mb x)
:=
\pm u^2v
+
v^{k-1}
+
x_1u
+
x_0
+
\sum_{i=2}^{k-1}x_i v^{i-1},
$$
where $\mb x=(x_0,\ldots,x_{k-1})\in \K^k$.

\begin{Remark}\label{rmk:433}
The sign ambiguity in the choice of $h$ affects the singularity type
only when $k$ is even.
Indeed, if $k$ is odd, then $k-1$ is even, so replacing $v$ with $-v$
leaves the leading term $v^{k-1}$ of $F$ unchanged while reversing
the sign of the term $\pm u^2v$.
Consequently, when $k$ is odd, a map germ of type $D_k^+$ is
right-left equivalent to one of type $D_k^-$.
By contrast, when $k$ is even, the two map germs are not
right-left equivalent.
\end{Remark}

\begin{Proposition}\label{prop:D_image}
The image of the standard map $\hat h$ coincides with
$$
\mathcal W_{D_k}(\K)
:=
\left\{
\mb x\in \K^k\,;\,
\text{there exists }(u,v)\in\K^2
\text{ such that }
F=F_u=F_v=0
\right\}.
$$
\end{Proposition}

\begin{proof}
The equation $F_u=0$ gives $x_1=h_1(u,v,\mb x_3)$ by
\eqref{eq:374b}.
Solving $F_v=0$ gives $x_2$ as in \eqref{eq:374c}.
Substituting these two relations into $F=0$ recovers
\eqref{eq:374a}.  The converse is immediate from the same computation.
\end{proof}

For $v\ne0$, substituting $u=\mp x_1/(2v)$
(cf. \eqref{eq:374b})
into $F$, we define
$$
  A^{\mb x}(v)
  :=
  v\,F\!\left(\mp\frac{x_1}{2v},v,\mb x\right).
$$
Then
\begin{equation}\label{eq:D_A}
A^{\mb x}(v)
=
v^k
+
x_{k-1}v^{k-1}
+\cdots
+
x_2v^2
+
x_0v
\mp\frac14 x_1^2,
\end{equation}
where the sign of the last term is opposite to the sign of the term
$\pm u^2v$ in $F$.
We set
$$
  B^{\mb x}(v):=\frac{dA^{\mb x}}{dv}.
$$

For later use, we set
\begin{equation}\label{eq:stst}
  \Theta_{D_k}
  :=
  \operatorname{Res}_v(A^{\mb x},B^{\mb x}),
\qquad
  s_0(\mb x)
  :=
  \operatorname{Psc}_v
  \bigl(A^{\mb x},B^{\mb x}\bigr).
\end{equation}

\begin{Proposition}\label{prop:D_basic}
The map $\hat h$ is proper.
For each $\mb x\in\K^k$, the inverse image
$\hat h^{-1}(\mb x)$ consists of at most $k$ points, and
$\hat h^{-1}(\mb 0)=\{o\}$, where $o$ $($resp. $\mb 0)$ denotes
the origin of $\K^{k-1}$ $($resp. $\K^k)$.  Moreover, $h$ is a
$(k-1)$-dimensional real analytic map.
\end{Proposition}

\begin{proof}
Let $\mb x=\hat h(u,v, \mb x_3)$ for $\mb x\in \K^k$.
Then by \eqref{eq:374b},
if $v\ne0$, then
$
  u=\mp x_1/(2v).
$
By the definition of $A^{\mb x}$, the corresponding $v$-values
satisfy
$
  A^{\mb x}(v)=0.
$
Let $\mc K$ be a compact subset of $\K^k$.  On $\mc K$, the
coefficients of $A^{\mb x}$ are bounded.  Hence, by
Remark~\ref{rmk:Cauchy},
all possible values of $v$ on
$\hat h^{-1}(\mc K)$ are uniformly bounded.
Moreover, by \eqref{eq:374c},
\[
\mp u^2
=
x_2+(k-1)v^{k-2}
+\sum_{i=3}^{k-1}(i-1)x_iv^{i-2}.
\]
Since $\mb x\in \mc K$ and $v$ is uniformly bounded, the
right hand side is uniformly bounded.  Hence $|u|$
is uniformly bounded.  Thus $\hat h^{-1}(\mc K)$ is bounded.
Since $\hat h^{-1}(\mc K)$ is closed, it is compact.  Therefore
$\hat h$ is proper.

Fix $\mb x\in\hat h(\K^{k-1})$.  For every
$(u,v,\mb x_3)\in\hat h^{-1}(\mb x)$, the value $v\in\K$ is a root
of $A^{\mb x}$.  Each non-zero root corresponds to at most one point
of $h^{-1}(\mb x)$, since $u=\mp x_1/(2v)$ by \eqref{eq:374b}.
If no point of $h^{-1}(\mb x)$ has $v=0$, then
$\#\hat h^{-1}(\mb x)\le k$.
Otherwise, \eqref{eq:374a} and \eqref{eq:374b} give $x_0=x_1=0$,
so $v=0$ is a root of $A^{\mb x}$ of multiplicity at least two.
There are therefore at most $k-2$ distinct non-zero roots.
Moreover, \eqref{eq:374c} gives at most two possible values of $u$
when $v=0$.
Hence
$$\#\hat h^{-1}(\mb x)\le(k-2)+2=k.
$$

Finally, $\hat h^{-1}(\mb 0)=\{o\}$ follows directly from
\eqref{eq:374a}--\eqref{eq:374c}.
To verify that $h$ is a $(k-1)$-dimensional real analytic map,
consider the submatrix of Jacobian matrix of $h$
corresponding to $(h_1,h_2)$:
At $(u,v,\mb x_3)=(1,0,\mb 0)$, its determinant is
\[
\det\begin{pmatrix}0&\mp2\\ \mp2&0\end{pmatrix}=-4.
\]
Since this is a nonzero polynomial, $\operatorname{rank}dh=k-1$
on an open dense subset of $\R^{k-1}$.
\end{proof}

\begin{Proposition}\label{new:L}
Let $\K=\R$ or $\C$
and fix $v\in \K\setminus \{0\}$.
For $\mb x\in\K^k$, there exists
$u\in\K$
such that
$$
  \hat h(u,v,\mb x_3)=\mb x,
  \qquad
  \mb x_3=(x_3,\ldots,x_{k-1})
$$
if and only if
$
  A^{\mb x}(v)=B^{\mb x}(v)=0.
$
\end{Proposition}

\begin{proof}
Since $v\ne0$, the equation $F_u=0$ gives
$
  u=\mp x_1/(2v).
$
Hence
$$
  A^{\mb x}(v)=vF(u,v,\mb x).
$$
Differentiating with respect to $v$, keeping $u=\mp x_1/(2v)$, gives
$$
  B^{\mb x}(v)
  =
  F(u,v,\mb x)+vF_v(u,v,\mb x).
$$
Therefore, for $v\ne0$, the system
$
  F=F_u=F_v=0
$
is equivalent to
$
  A^{\mb x}=B^{\mb x}=0
$
with $u=\mp x_1/(2v)$.
This with Proposition \ref{prop:D_image}
proves the assertion.
\end{proof}

\begin{Theorem}\label{thm:D_main}
The image
$
h(\R^{k-1})
$
is a global main-analytic set in $\R^k$, and
$\Theta_{D_k}$
given in
\eqref{eq:stst}
is a main-analytic function for it.
\end{Theorem}

\begin{proof}
We first show
$
  h(\R^{k-1})\subset \mathcal Z(\Theta_{D_k}).
$
Fix $\mb x\in h(\R^{k-1})$. We can write
$\mb x=h(u,v,\mb x_3)$ for $u,v\in \R$.
If $v\ne0$, then
Proposition~\ref{new:L} implies that $A^{\mb x}$ and
$B^{\mb x}$ have a common root.  Hence
$
  \Theta_{D_k}(\mb x)=0.
$
On the other hand, if $v=0$, then \eqref{eq:374a} and
\eqref{eq:374b} give
$
  x_0=x_1=0.
$
Thus
$
  A^{\mb x}(0)=B^{\mb x}(0)=0,
$
so again $\Theta_{D_k}(\mb x)=0$.

We next estimate the dimension of
 $\mathcal Z(\Theta_{D_k})\setminus h(\R^{k-1})$.
Set
$$
  \mathcal L
  :=
  \left\{
  \mb x\in \mathcal Z(\Theta_{D_k})\,;\,
  \text{there exists }v\in\C\setminus\R
  \text{ such that }
  A^{\mb x}(v)=B^{\mb x}(v)=0
  \right\}.
$$
We claim that $h(=h^\pm_{D_k})$ satisfies
\begin{equation}\label{eq:D-difference-contained}
  \mathcal Z(\Theta_{D_k})\setminus h^\pm_{D_k}(\R^{k-1})
  \subset
  \mathcal L\cup\mathcal E_\pm,
\end{equation}
where
$$
  \mathcal E_+
  =
  \{x_0=x_1=0,\ x_2>0\},
  \quad
  \mathcal E_-
  =
  \{x_0=x_1=0,\ x_2<0\}.
$$
In fact, let
$
  \mb x\in \mathcal Z(\Theta_{D_k})\setminus h(\R^{k-1}).
$
Since $\Theta_{D_k}(\mb x)=0$, the polynomials
$A^{\mb x}$ and $B^{\mb x}$ have a common root in $\C$.
If they have a non-real common root, then $\mb x\in\mathcal L$.
If they have a real non-zero common root, Proposition~\ref{new:L}
implies
$
  \mb x\in h(\R^{k-1}),
$
which is a contradiction.  Hence the only possible real common root is
$v=0$.  In that case
$
  A^{\mb x}(0)=B^{\mb x}(0)=0,
$
and so
$
  x_0=x_1=0.
$
At $v=0$, equation \eqref{eq:374c} gives
$
  x_2=\mp u^2.
$
Thus, a point with $x_0=x_1=0$ is realized with $v=0$ if and only if
$
  \pm x_2\le0.
$
Since $\mb x\notin h(\R^{k-1})$, we must have
$
  \pm x_2>0,
$
that is, $\mb x\in\mathcal E_\pm$.  This proves
\eqref{eq:D-difference-contained}.

\medskip
It remains to estimate dimensions.  If $\alpha$ is a non-real common
root of $A^{\mb x}$ and $B^{\mb x}$, then $\bar\alpha$ is
also a common root.  Therefore, we can write
$$
  A^{\mb x}(v)
  =
  \bigl(v^2-(\alpha+\bar\alpha)v+|\alpha|^2\bigr)^2 Q(v),
  \qquad
  \deg_v Q=k-4,
$$
and $Q(v)$ is a monic polynomial.
Thus
$$
  \dim_H\mathcal L\le 2+(k-4)=k-2.
$$
On the other hand, it is obvious that
$
  \dim_H\mathcal E_\pm=k-2.
$
Since $h$ is a $(k-1)$-dimensional real analytic map, we have
$
  \dim_H h(\R^{k-1})=k-1.
$
Moreover, since $\Theta_{D_k}$ is a non-zero polynomial on $\R^k$,
we have $\dim_H\mathcal Z(\Theta_{D_k})\le k-1$.
Since $h(\R^{k-1})\subset\mathcal Z(\Theta_{D_k})$ and
$\dim_H h(\R^{k-1})=k-1$,
we have $\dim_H\mathcal Z(\Theta_{D_k})=k-1$.
The inclusion \eqref{eq:D-difference-contained} shows that
$\mathcal Z(\Theta_{D_k})$ differs from $h(\R^{k-1})$ only by a
subset of Hausdorff dimension at most $k-2$.  Hence
$\Theta_{D_k}$ is a main-analytic function for $h(\R^{k-1})$.
\end{proof}

\begin{Proposition}\label{prop:D_complex_irred}\label{rmk:d}
For the $D_k$-case, one has
$
  h^\C(\C^{k-1})=\mathcal Z_\C(\Theta_{D_k}).
$
Moreover, $\Theta_{D_k}$ is irreducible over $\C$.
Consequently, $\Theta_{D_k}$ coincides, up to a non-zero scalar,
with the discriminant polynomial of type $D_k$.
\end{Proposition}

\begin{proof}
Let
$\mb x\in h^\C(\C^{k-1})$.  If the corresponding source point has
$v\ne0$, then Proposition~\ref{new:L} shows that
$\Theta_{D_k}(\mb x)=0$.  If $v=0$, then \eqref{eq:374a} and
\eqref{eq:374b} give $x_0=x_1=0$.  Hence
$A^{\mb x}(0)=B^{\mb x}(0)=0$, so again
$\Theta_{D_k}(\mb x)=0$. Thus, we have
$h^\C(\C^{k-1})\subset \mathcal Z_\C(\Theta_{D_k})$.

Conversely, let $\mb x\in\mathcal Z_\C(\Theta_{D_k})$.  Then
$A^{\mb x}$ and $B^{\mb x}=dA^{\mb x}/dv$ have a common root
$v_0\in\C$.  If $v_0\ne0$, Proposition~\ref{new:L} gives
$\mb x\in h^\C(\C^{k-1})$.  If $v_0=0$, then
$A^{\mb x}(0)=B^{\mb x}(0)=0$, hence $x_0=x_1=0$.  Since we work over
$\C$, we can choose $u\in\C$ with $x_2=\mp u^2$.  Then
\eqref{eq:374a}--\eqref{eq:374c} imply
$\mb x=h^\C(u,0,\mb x_3)$.  Thus
$
  h^\C(\C^{k-1})=\mathcal Z_\C(\Theta_{D_k}).
$

We next prove irreducibility.  Since $\C^{k-1}$ is irreducible and
$h^\C$ is a polynomial map, $h^\C(\C^{k-1})$ is irreducible.  Hence
$\mathcal Z_\C(\Theta_{D_k})$ is irreducible.  Therefore,
there exists an integer $m\ge1$
such that
$$
  \Theta_{D_k}=cP^m \qquad
$$
where $c\in\C\setminus\{0\}$ and $P\in\C[x_0,\ldots,x_{k-1}]$
is an irreducible polynomial.
It remains to
show that $m=1$.
For a polynomial
$
  p(v)=\alpha \prod_{i=1}^k(v-\rho_i)
$ ($\alpha\ne 0$),
we define
$$
  \operatorname{Disc}_v(p)
  :=
  \alpha^{2k-2} \prod_{i<j}(\rho_i-\rho_j)^2,
$$
which is the discriminant of the polynomial $p(v)$, and so,
$\operatorname{Disc}_v(p)=0$ if and only if $p$ has a multiple root.
For monic polynomial $q$ of degree $k$,
\begin{equation}\label{eq:1376}
  \operatorname{Res}_v(q,dq/dv)
  =
  (-1)^{k(k-1)/2}\operatorname{Disc}_v(q).
\end{equation}
Hence $\Theta_{D_k}(\mb x)=\operatorname{Res}_v(A^{\mb x},dA^{\mb x}/dv)$
differs from the discriminant of $A^{\mb x}$ only by this fixed non-zero
sign.
Let $\varepsilon=\mp1/4$ so that
$$
  A^{\mb x}(v)
  =
  v^k+x_{k-1}v^{k-1}+\cdots+x_2v^2+x_0v+\varepsilon x_1^2.
$$
Choose
a monic polynomial
$G(v)\in\R[v]$
of degree $k-2$ such that
$\varepsilon G(0)>0$
and the product
$G(1)\operatorname{Disc}_v(G)$ does not vanish.
For a parameter $t\in \R$, we set
$$
  p_t(v):=((v-1)^2+t)G(v).
$$
Since $G(1)\ne 0$, the polynomial
$p_0$ has exactly one double root at $v=1$.  Using
$$
  \operatorname{Disc}_v(fg)
  =
  \operatorname{Disc}_v(f)\operatorname{Disc}_v(g)
  \operatorname{Res}_v(f,g)^2
$$
for monic polynomials $f$ and $g$, we get
\begin{equation}\label{eq:1373}
\begin{aligned}
  \operatorname{Disc}_v(p_t)
  &=
  \operatorname{Disc}_v((v-1)^2+t)\operatorname{Disc}_v(G)
  \operatorname{Res}_v((v-1)^2+t,G)^2  \\
  &=
  (-4t)\operatorname{Disc}_v(G)\bigl(G(1)^2+O(t)\bigr)^2
   = C_1t+O(t^2),
\end{aligned}
\end{equation}
where $C_1\ne0$.
Write
$$
  p_t(v)=v^k+c_{k-1}(t)v^{k-1}+\cdots+c_1(t)v+c_0(t).
$$
Since
$
  c_0(0)/\varepsilon=G(0)/\varepsilon>0,
$
the function
$
x_1(t):=\sqrt{c_0(t)/\epsilon}
$
is well-defined as a smooth function
defined on a neighborhood of $t=0$.
Set
$$
\mb x(t):=(c_1(t),x_1(t),c_2(t),\ldots,c_{k-1}(t)).
$$
Then $A^{\mb x(t)}(v)=p_t(v)$ and,
by our definition of $p_t(v)$,
we have
$$
\mb x(0)=h(u_0,1,c_3(0),\ldots, c_{k-1}(0))\in h(\R^{k-1})
$$
for some real number $u_0$.
By \eqref{eq:1376} and \eqref{eq:1373},
we have
\begin{equation}\label{eq:1407}
  \Theta_{D_k}(\mb x(t))
  =
  \operatorname{Res}_v(p_t,dp_t/dv)
  =
  C_2t+O(t^2)
  \qquad (C_2\ne0).
\end{equation}
Since
$cP^m(\mb x(0))=\Theta_{D_k}(\mb x(0))=0$,
we have $P(\mb x(0))=0$.
If $m\ge2$, then
$P(\mb x(t))^m$ vanishes to order
at least two at $t=0$.  This contradicts
\eqref{eq:1407}.
Hence $m=1$, and $\Theta_{D_k}$ is irreducible over $\C$.

Consequently, $\Theta_{D_k}$ is, up to a non-zero scalar, the defining
equation of the irreducible hypersurface $h^\C(\C^{k-1})$.  Hence it
coincides, up to a non-zero scalar, with the discriminant
polynomial of type $D_k$.
\end{proof}

\subsection*{Proof of Theorem~C for type $D$}
Recall that
$$
  s(\mb x):=(x_0^2+x_1^2)s_0(\mb x),
  \qquad
  s_0(\mb x):=
  \operatorname{Psc}_v\bigl(A^{\mb x},B^{\mb x}\bigr).
$$
Thus,
\begin{equation}\label{eq:1474}
  \mathcal Z(s)=\mathcal Z(s_0)\cup\{x_0=x_1=0\}.
\end{equation}

\begin{proof}[Proof of the inclusion in {\rm (1)}]
We suppose that there exists $\mb x\in \R^k$ such that
$$
  h(u_1,v_1,\mb x_3)=h(u_2,v_2,\mb x_3)=\mb x
$$
for some $u_1,u_2,v_1,v_2\in \R$  with $(u_1,v_1)\ne(u_2,v_2)$.
 If one of the
$v_i$ is zero, then \eqref{eq:374a} and \eqref{eq:374b} give
$x_0=x_1=0$, and hence $s(\mb x)=0$.  Suppose that both $v_i$ are
non-zero.  If $v_1=v_2$, then $x_1=\mp2uv$ determines $u_1=u_2$, so
the two source points coincide, a contradiction.  Hence $v_1\ne v_2$.
By the computation in the proof of Proposition~\ref{new:L},
we have $A^{\mb x}(v_i)=B^{\mb x}(v_i)=0$
$(i=1,2)$.
Thus $A^{\mb x}$ and $B^{\mb x}$ have two distinct common roots, so
$s_0(\mb x)=0$, and again
$s(\mb x)=0$.
\end{proof}

\begin{proof}[Proof of the non-inclusion in {\rm (1)}]
Let $h=h_{D_k}^{\pm}$, with $k\ge4$, and choose $a>0$.
By \eqref{eq:374a}--\eqref{eq:374c},
\[
  h(a,0,\mb 0)=h(-a,0,\mb 0)
  =\mb y:=(0,0,\mp a^2,0,\ldots,0).
\]
Since $(a,0,\mb 0)\ne(-a,0,\mb 0)$, we have $\mb y\in h(\mathcal I_h)$.
On the other hand, \eqref{eq:D_A} gives
\[
  A^{\mb y}(v)=v^k\mp a^2v^2,\qquad
  B^{\mb y}(v)=kv^{k-1}\mp2a^2v.
\]
Since $vB^{\mb y}-2A^{\mb y}=(k-2)v^k$, their only common root is $0$,
which is a simple root of $B^{\mb y}$.
Thus $\gcd(A^{\mb y},B^{\mb y})=v$, and
Proposition~\ref{prop:appA-subresultant}(2) yields $s_0(\mb y)\ne0$.
Consequently, $h(\mathcal I_h)\not\subset\mathcal Z(s_0)$.
\end{proof}

\begin{proof}[Proof of \rm (2) and the last statement of \rm (1)]
By \eqref{eq:1474},
we have
$$
\mathcal Z(s)\cap\mathcal Z(\Theta_{D_k})
\subset
\bigl(\mathcal Z(s_0)\cap\mathcal Z(\Theta_{D_k})\bigr)
\cup\{x_0=x_1=0\}.
$$
The set $\{x_0=x_1=0\}$ has dimension $k-2$.  On
$\mathcal Z(s_0)\cap\mathcal Z(\Theta_{D_k})$, the polynomials
$A^{\mathbf x}$ and $B^{\mathbf x}(=dA^{\mathbf x}/dv)$ have a common divisor of
degree at least two.
Thus $A^{\mathbf x}$ has either two distinct multiple
roots or a root of multiplicity at least three.
Then, by the same argument as in the proof of Theorem C for type A,
the dimension of
$\mathcal Z(s_0)\cap\mathcal Z(\Theta_{D_k})$
is at
most $k-2$.
This proves (2) of Theorem~C for type $D$.
Moreover, by the same argument as in the proof of
Theorem~C for type $A$, the map $h$ is injective on
an open dense subset of $\R^{k-1}$, which proves
the last statement of (1).
\end{proof}

\begin{proof}[Proof of the first statement of \rm (3)]
Set
$$
  \Delta_D:=\mathcal Z(\Theta_{D_k})
  \setminus h_{D_k}^{\pm}(\R^{k-1}).
$$
Let $\mb x\in\Delta_D$.  If $s_0(\mb x)=0$, then already
$\mb x\in\mathcal Z(s)$.  So we may assume $s_0(\mb x)\ne0$.  Since
$\Theta_{D_k}(\mb x)=0$, the polynomials $A^{\mb x}$ and $B^{\mb x}$
have a common root.
Moreover,  the condition $s_0(\mb x)\ne0$ says that their
common divisor has degree one.  In particular, the common root is
unique.  It must be real, because a non-real common root would bring
its conjugate.  If the common root were non-zero,
Proposition~\ref{new:L} would imply
$\mb x\in h_{D_k}^{\pm}(\R^{k-1})$, a contradiction.  Hence the unique
common root is $v=0$.  Therefore
$
  A^{\mb x}(0)=B^{\mb x}(0)=0.
$
By \eqref{eq:D_A}, this means
$\mp x_1^2/4$ and $x_0$ vanish.
So, we have $x_0=x_1=0$.  Hence
$\mb x\in\mathcal Z(s)$, and so
$
  \Delta_D\subset\mathcal Z(s).
$
\end{proof}

\begin{proof}[Proof of the last statement of \rm (3)]
Finally, we show that $\Delta_D$ is not a subset of
$\mathcal Z(s_0)$ (cf. \eqref{342}).
 Choose
$\epsilon\in\R\setminus\{0\}$ with $\pm\epsilon>0$, and set
$$
(\mb y=)(y_0,\ldots,y_{k-1}):=(0,0,\epsilon,0,\ldots,0).
$$
Then
$$
A^{\mb y}(v)=v^k+\epsilon v^2,
\qquad
B^{\mb y}(v)=kv^{k-1}+2\epsilon v.
$$
The only common root of these two polynomials is $v=0$.
Moreover, $v=0$ is a simple root of
$B^{\mb y}$, since $\epsilon\ne0$.  Hence
$\Theta_{D_k}(\mb y)=0$ but $s_0(\mb y)\ne0$.

Also, $\mb y\notin h_{D_k}^{\pm}(\R^{k-1})$.  Indeed, a representation
with $v\ne0$ is impossible by Proposition~\ref{new:L}.  If $v=0$, then the
formula for $h_{D_k}^{\pm}$ gives
$\epsilon=y_2=\mp u^2$, contradicting
$\pm\epsilon>0$.  Therefore
$$
\mb y\in
\bigl(\mathcal Z(\Theta_{D_k})
\setminus h_{D_k}^{\pm}(\R^{k-1})\bigr)
\setminus\mathcal Z(s_0)=\Delta_D\setminus\mathcal Z(s_0),
$$
which implies
$
\Delta_D
\not\subset\mathcal Z(s_0).
$
\end{proof}

\subsection*{Proof of Corollary B for type $D$}

For $\mb x\in \R^k$, set
$$
\Sigma(\mb x)
:=
\left\{
v\in \R\setminus\{0\}\,;\,
A^{\mb x}(v)=B^{\mb x}(v)=0
\right\}.
$$
Then Proposition~\ref{new:L} yields a bijection
$$
\Sigma(\mb x)\ni v
\mapsto
\left(\mp \frac{x_1}{2v},\,v,\,\mb x_3\right)
\in h^{-1}(\mb x)\cap\{v\neq 0\}.
$$
On the other hand, by \eqref{eq:374a}--\eqref{eq:374c},
the intersection
$h^{-1}(\mb x)\cap\{v=0\}$ is non-empty if and only if
$$
  x_0=x_1=0,
  \qquad
  \mp x_2(=u^2)\ge 0.
$$
Hence, if we set
$$
\eta(\mb x)
:=
\#\bigl(h^{-1}(\mb x)\cap\{v=0\}\bigr),
$$
then
$$
\eta(\mb x)=
\begin{cases}
2 & \text{if } h=h^+_{D_k},\ x_0=x_1=0,\ x_2<0,\\
2 & \text{if } h=h^-_{D_k},\ x_0=x_1=0,\ x_2>0,\\
1 & \text{if } x_0=x_1=x_2=0,\\
0 & \text{otherwise.}
\end{cases}
$$
In particular,
\begin{equation}\label{eq:Dfiber}
\# h^{-1}(\mb x)=\#\Sigma(\mb x)+\eta(\mb x).
\end{equation}

\begin{proof}[Proof of \rm (2)]
We first consider the case $k=2m+1$, where $m\ge2$.
By Remark~\ref{rmk:433}, we may assume that $h=h^-_{D_k}$, and so
$$
A^{\mb x}(v)
=
v^k+x_{k-1}v^{k-1}+\cdots+x_2v^2+x_0v+\frac14 x_1^2.
$$
If $\eta(\mb x)=0$, then $A^{\mb x}(v)$ has at most $m$
distinct non-zero real multiple roots, since $\deg_v
A^{\mb x}=2m+1$.
Hence \eqref{eq:Dfiber} gives
$
  \#h^{-1}(\mb x)\le m.
$
If $\eta(\mb x)\ne0$, then $x_0=x_1=0$, and so $v=0$ is a
root of multiplicity at least two of $A^{\mb x}(v)$.  Thus $A^{\mb x}(v)$ has at most
$m-1$ distinct non-zero real multiple roots, and \eqref{eq:Dfiber}
gives
$$
  \#h^{-1}(\mb x)\le (m-1)+2=m+1.
$$
Therefore
$
  M\le m+1.
$
To see the equality, choose a sufficiently small
positive number $\lambda_0$ and
sufficiently small distinct non-zero real numbers
$
  \lambda_1,\ldots,\lambda_{m-1}
$
such that $\lambda_i\ne-\lambda_0$ for all $i$.  Let
$\mb x=(x_0,\ldots,x_{k-1})\in\R^k$ be determined by
$$
A^{\mb x}(v)
=
v^2(v+\lambda_0)\prod_{i=1}^{m-1}(v-\lambda_i)^2.
$$
Then we may assume that $|\mb x|<\epsilon$.
By definition, it holds that
$x_0=x_1=0$, and
$
  x_2=\lambda_0\prod_{i=1}^{m-1}\lambda_i^2>0.
$
Since $h=h^-_{D_k}$, equation \eqref{eq:374c} gives
$$
  (\pm\sqrt{x_2},0,\mb x_3)\in h^{-1}(\mb x).
$$
Moreover, each $\lambda_i$ is a non-zero double root of
$A^{\mb x}(v)$.  Since $x_1=0$, Proposition~\ref{new:L} yields
$$
  (0,\lambda_i,\mb x_3)\in h^{-1}(\mb x)
  \qquad (i=1,\ldots,m-1).
$$
Thus
$
  \#h^{-1}(\mb x)=m+1,
$
and so $M=m+1$.
\end{proof}

Next, let $k=2m$, where $m\ge2$.

\begin{Lemma}\label{lem:1633}
In both cases
$h=h^+_{D_k}$ and $h=h^-_{D_k}$,
the following hold:
\begin{enumerate}
\item[{\rm (i)}]
If $\eta(\mb x)=0$, then
$
  \#h^{-1}(\mb x)\le m.
$
\item[{\rm (ii)}]
If $\eta(\mb x)=1$, then
$
  \#h^{-1}(\mb x)\le m-1
$.
\end{enumerate}
\end{Lemma}

In particular, to prove
(3) and (4) of Corollary B, it is
sufficient to consider the case that $\eta(\mb x)=2$.

\begin{proof}
If $\eta(\mb x)=0$, then \eqref{eq:Dfiber} gives
the assertion.
If $\eta(\mb x)=1$, then $x_i=0$ for $i=0,1,2$, and hence $v=0$
is at least a triple root of $A^{\mb x}(v)$.  Therefore
$A^{\mb x}(v)$ has at most $m-2$ distinct non-zero
real multiple roots,
and
$$
  \#h^{-1}(\mb x)\le (m-2)+1=m-1.
$$
\end{proof}

\begin{proof}[Proof of \rm (3)]
Let $k=2m$ and $h=h^-_{D_{2m}}$.
By Lemma \ref{lem:1633},
it remains to consider the case
$\eta(\mb x)=2$.
Then $x_0=x_1=0$ and $x_2>0$, so $v=0$ is a double root of
$A^{\mb x}$.  Since $\deg_v A^{\mb x}=2m$,
the polynomial $A^{\mb x}$ has at most $m-1$ distinct non-zero
real multiple roots.  Hence, by \eqref{eq:Dfiber},
$$
  \#h^{-1}(\mb x)\le(m-1)+2=m+1.
$$
Choose sufficiently small distinct non-zero real numbers
$
  \lambda_1,\ldots,\lambda_{m-1},
$
and let $\mb x\in\R^k$ be determined by
$$
A^{\mb x}(v)
=
v^2\prod_{i=1}^{m-1}(v-\lambda_i)^2.
$$
Then $\mb x$ can be chosen so that $|\mb x|<\epsilon$.
Moreover,
$$
x_0=x_1=0,
\qquad
x_2=\prod_{i=1}^{m-1}\lambda_i^2>0.
$$
It follows that
$
  (\pm\sqrt{x_2},0,\mb x_3)\in h^{-1}(\mb x).
$
Since each $\lambda_i$ is a non-zero double root of $A^{\mb x}$,
Proposition~\ref{new:L}, together with $x_1=0$, gives
$$
  (0,\lambda_i,\mb x_3)\in h^{-1}(\mb x)
  \qquad (i=1,\ldots,m-1).
$$
Thus
$$
  \#h^{-1}(\mb x)=2+(m-1)=m+1,
$$
and hence $M=m+1$.
\end{proof}

\begin{proof}[Proof of {\rm (4)}]
Let $k=2m$ and $h=h^+_{D_{2m}}$.
By Lemma \ref{lem:1633},
it remains to consider the case
$\eta(\mb x)=2$.  Then $x_0=x_1=0$ and $x_2<0$, so $v=0$
is a double root of $A^{\mb x}$.

Suppose that $A^{\mb x}$ has $m-1$ distinct non-zero real multiple
roots.  Since $\deg_v A^{\mb x}=2m$, all these roots must be double,
and hence
$$
  A^{\mb x}(v)
  =
  v^2\prod_{i=1}^{m-1}(v-\lambda_i)^2
$$
for some distinct non-zero real numbers
$\lambda_1,\ldots,\lambda_{m-1}$.  It would then follow that
$$
  x_2=\prod_{i=1}^{m-1}\lambda_i^2>0,
$$
contrary to $x_2<0$.  Thus $A^{\mb x}$ has at most $m-2$
distinct non-zero real multiple roots.  By \eqref{eq:Dfiber},
$$
  \#h^{-1}(\mb x)\le (m-2)+2=m.
$$
Therefore $M\le m$.

To show that this bound is sharp, choose a sufficiently small
positive number $\lambda_0$ and, if $m>2$, sufficiently small
distinct non-zero real numbers
$
  \lambda_1,\ldots,\lambda_{m-2}
$
such that $\lambda_i\ne\pm\lambda_0$ for every $i$.
Let $\mb x\in\R^k$ be determined by
\begin{equation}\label{eq:1826}
  A^{\mb x}(v)
  =
  v^2(v^2-\lambda_0^2)
  \prod_{i=1}^{m-2}(v-\lambda_i)^2,
\end{equation}
where the product is understood to be $1$ when $m=2$.
By choosing the $\lambda_i$ sufficiently small, we may assume that
$|\mb x|<\epsilon$.  From the right hand side of
\eqref{eq:1826},
we have
$$
  x_0=x_1=0,
  \qquad
  x_2
  =
  -\lambda_0^2
  \prod_{i=1}^{m-2}\lambda_i^2
  <0.
$$
Hence
$
  (\pm\sqrt{-x_2},0,\mb x_3)\in h^{-1}(\mb x).
$
Moreover, each $\lambda_i$ is a non-zero double root of
$A^{\mb x}$.  Since $x_1=0$, Proposition~\ref{new:L} gives
\begin{equation}\label{eq:ss}
  (0,\lambda_i,\mb x_3)\in h^{-1}(\mb x)
  \qquad (i=1,\ldots,m-2).
\end{equation}
\emergencystretch=2em
The only non-zero multiple roots of $A^{\mb x}$ are the
$\lambda_i$, while $v=0$ gives the two points
$(\pm\sqrt{-x_2},0,\mb x_3)$.
Thus, by \eqref{eq:Dfiber},
$$
  \#h^{-1}(\mb x)=2+(m-2)=m.
$$
Therefore $M=m$ for $h=h^+_{D_{2m}}$.
\end{proof}

\section{Singularities of type $E_6$}

We shall use the following elementary criterion for the
irreducibility,
which will apply in the proof of Proposition \ref{prop:Z}
below.

\begin{Lemma}\label{lem:rj-irred-criterion}
Let $R$ be a UFD and let
$K$ be the field of fractions of $R$.
\begin{enumerate}
\item Every primitive polynomial in $R[z]$ of degree one in $z$ is
irreducible in $R[z]$.
\item Let
$$
F=az^2+bz+c\in R[z]\qquad (a\ne 0)
$$
be primitive.  If $b^2-4ac$ is not a square in $K$,
then $F$ is irreducible in $R[z]$.
\end{enumerate}
\end{Lemma}

\begin{proof}
By Gauss' lemma, it suffices to work over $K$.
There, {\rm(1)} is immediate, and reducibility in {\rm(2)} forces
the discriminant to be a square.
\end{proof}

As in Section~1, the standard map
$
h=h_{E_6}:\R^5\to \R^6
$
is defined by
$$
h(u,v,\mb x_3)=
\bigl(h_0,h_1,h_2,\mb x_3\bigr),
$$
where $\mb x_3=(x_3,x_4,x_5)$ and
$h_i$ ($i=0,1,2$) are given in
\eqref{eq:1022a},
\eqref{eq:1022b} and
\eqref{eq:1022c}.
We write
$$
d^{\mb x}(v):=d^{\mb x}_6(v)(=x_4+2x_5v).
$$
Set $\K=\R$ or $\C$, and let
$
  \hat h:\K^5\to\K^6
$
denote $h$ when $\K=\R$ and $h^\C$ when $\K=\C$.

\begin{Proposition}\label{prop:e6}
The image of the standard map $\hat h$ coincides with
the set
\begin{align*}
(\mc W_{E_6}(\K)&:=)\,\Big\{\mb x:=(x_0,\ldots,x_5)\in \K^6\,;\,
\text{there exists $(u,v)\in \K^2$ such that}
\\
& \phantom{aaaaaaaaaaaaaaaaaaaaaa}
\text{$F(u,v,\mb x)=F_u(u,v,\mb x)=F_v(u,v,\mb x)=0$}\Big\},
\end{align*}
where
$
F=u^3+v^4+x_5uv^2+x_4uv+x_3v^2+x_2v+x_1u+x_0.
$
\end{Proposition}

\begin{proof}
From $F_u=F_v=0$ we obtain \eqref{eq:1022b} and \eqref{eq:1022c}.
Substituting these into $F=0$ yields \eqref{eq:1022a}.
\end{proof}

We note that $F_{uv}=d^{\mb x}$ holds.
A direct Jacobian computation gives the following statement:

\begin{Proposition}\label{prop:1101}
The map $h$ is a $5$-dimensional real analytic map.
\end{Proposition}

Corresponding to \eqref{eq:1022a}, \eqref{eq:1022b} and \eqref{eq:1022c},
for each $\mb x:=(x_0,\ldots,x_5)\in \R^6$,  we set
\begin{align}\label{eq:1022a1}
g_0^{\mb x}(u,v)&:=2u^3+3v^4+x_3v^2-x_0+uv\, d^{\mb x}(v), \\
\label{eq:1022b1}
g_1^{\mb x}(u,v)&:=3u^2+x_4v+x_5v^2+x_1, \\
\label{eq:1022c1}
g_2^{\mb x}(u,v)&:=-4 v^3-2x_3v-x_2-u\,d^{\mb x}(v).
\end{align}

\begin{Proposition}\label{prop:E1057}
With the notation above, the map $\hat h:\K^5\to\K^6$ is proper.
Moreover, $\hat h^{-1}(\mb x)$ is finite for every $\mb x$,
and $\hat h^{-1}(\mb 0)=\{o\}$.
\end{Proposition}

\begin{proof}
The equations at $\mb x=\mb0$ give $u=v=0$.
For any $\hat h(u,v,\mb x_3)=\mb x$, we have
\begin{equation}\label{eq:194a}
u^2=\frac{-x_1-x_4v-x_5v^2}{3},
\end{equation}
and $u d^{\mb x}(v)=-(4v^3+2x_3v+x_2)$. Thus, when $d^{\mb x}(v)\ne0$,
\begin{equation}\label{eq:1034}
u=-\frac{4v^3+2x_3v+x_2}{d^{\mb x}(v)}.
\end{equation}
Squaring the identity for $u d^{\mb x}(v)$ and using \eqref{eq:194a}
gives $A^{\mb x}(v)=0$ for every preimage, including $d^{\mb x}(v)=0$, where
\begin{equation}\label{eq:Ad}
A^{\mb x}(v)
=3(4v^3+2x_3v+x_2)^2
+d^{\mb x}(v)^2(x_1+x_4v+x_5v^2).
\end{equation}
Its leading term is $48v^6$, so $v$ has finitely many possible values.
Thus, by \eqref{eq:194a}, $u$ also has finitely many possible values.
On the inverse image of a compact set, the coefficients of $A^{\mb x}$
are uniformly bounded. The Cauchy bound (Remark~\ref{rmk:Cauchy}) bounds
$v$, and \eqref{eq:194a} bounds $u$. Since $\mb x_3$ is part of the target,
this inverse image is closed and bounded, hence compact over either
$\R$ or $\C$.
\end{proof}

Substituting \eqref{eq:194a} and \eqref{eq:1034}
into the equation $g_0^{\mb x}(u,v)=0$, we obtain
the equation
$B^{\mb x}(v)=0$, where
\begin{align}
B^{\mb x}(v)&:=2 x_5v^5  +5 x_4v^4 +(8 x_1-2 x_3 x_5)v^3
+(x_3 x_4-4 x_2x_5)v^2  \\
\nonumber
&\phantom{aaaaaaaaaaaaa}+(-6 x_0 x_5+4 x_1 x_3-x_2 x_4)v
-3 x_0 x_4+2x_1 x_2.
\end{align}

If $\hat h(u,v,\mb x_3)=\mb x$,
then $A^{\mb x}(v)$ and $B^{\mb x}(v)$
have a common root in $\K$.
We set
\begin{equation}\label{R6}
\mc R_6(\mb x):=\op{Res}_v(A^{\mb x},B^{\mb x}),\qquad
s_6(\mb x):=\op{Psc}_v(A^{\mb x},B^{\mb x}).
\end{equation}
Moreover,
\begin{equation}\label{eq:r6}
r_6(\mb x):=-\frac14\op{Res}_v(g_2^{\mb x},d^{\mb x})
=x_4^3+2x_3x_5^2x_4-2x_2x_5^3,
\end{equation}

\begin{Proposition}\label{prop:Z}
The polynomial $r_6(\mb x)$ is irreducible over $\C$.
\end{Proposition}

\begin{proof}
Regard $r_6$ as a polynomial in $x_2$ over
$A=\C[x_3,x_4,x_5]$:
$$
r_6=-2x_5^3x_2+x_4^3+2x_3x_5^2x_4.
$$
Its two coefficients are coprime in $A$, since the only possible common
factor of $-2x_5^3$ and $x_4^3+2x_3x_5^2x_4$ would be a power of $x_5$,
but the latter polynomial is not divisible by $x_5$.  Hence $r_6$ is
primitive and degree one in $x_2$.  By Lemma~\ref{lem:rj-irred-criterion},
$r_6$ is irreducible in $\C[x_2,x_3,x_4,x_5]$.
\end{proof}

\begin{Proposition}\label{prop:1366}
Fix $\mb x\in\K^6$.
If $\mb x\in\hat h(\K^5)$,
then $A^{\mb x}$ and $B^{\mb x}$ have a common root in $\K$.
Conversely, if they have a common root $v\in\K$
with $d^{\mb x}(v)\ne0$,
then $\mb x\in\hat h(\K^5)$.
\end{Proposition}

\begin{proof}
The first assertion follows from the elimination.
By setting
$$
G_0:=g_0, \qquad G_1:=-g_1,\qquad G_2:=g_2,\qquad d:=d^{\mb x},
$$
the polynomials $\alpha$ and $\beta$ defined in Appendix~C
correspond to the polynomials $A^{\mb x}$ and $B^{\mb x}$.
Thus, the converse statement follows from Appendix~C
(cf. Proposition~\ref{thm:Cm}).
\end{proof}

\begin{Corollary}\label{eq:1366}
$h(\R^5)\subset\mc Z(\mc R_6)$.
\end{Corollary}

Considering $\mc R_6$ and $s_6$
as polynomials in $x_0$, we have
\begin{align}\label{eq:1284a}
\mc R_6(\mb x)
&=2^{20}3^{11}r_6^2x_0^6
+\text{(lower-order terms in $x_0$)}, \\
s_6(\mb x)
&=-2^{21}3^9x_5^5x_0^5
+\text{(lower-order terms in $x_0$)}. \label{S7lead}
\end{align}

\begin{Theorem}\label{thm:E6}
The set $\mc Z(\mc R_6)$ has dimension $5$, and
\[
 \dim_H\bigl(\mc Z(\mc R_6)\cap\mc Z(s_6)\bigr)\le4.
\]
The set
\begin{equation}\label{eq:E1X}
 \Xi_{E_6}:=\mc Z(\mc R_6)\setminus\mc Z(s_6)
\end{equation}
is open in $\mc Z(\mc R_6)$.
For each $\mb x\in\Xi_{E_6}$, the polynomials $A^{\mb x}$ and
$B^{\mb x}$ have exactly one common root in $\C$, and the root is real.
\end{Theorem}

\begin{proof}
Since $\mc R_6(\mb x)$ is non-constant and
$h(\R^5)\subset \mc Z(\mc R_6)$, Proposition~\ref{prop:1101} gives
\begin{equation}\label{eq:2058}
\dim_H \mc Z(\mc R_6)=5.
\end{equation}
By Propositions~\ref{prop:appE-square} and~\ref{prop:appE-irred-Q}
in Appendix~\ref{app:exceptional}, we have
$\mc R_6=r_6^2\Theta_{E_6}$, where $\Theta_{E_6}$ is irreducible over $\Q$.
We now show that neither $r_6$ nor $\Theta_{E_6}$ divides $s_6$:
By \eqref{eq:1284a} and \eqref{S7lead}, the degrees of $\Theta_{E_6}$ and
$s_6$ in $x_0$ are $6$ and $5$. Hence $\Theta_{E_6}$ cannot divide $s_6$.

The polynomial $r_6$ is independent of $x_0$. If $r_6$ divided $s_6$,
then $r_6$ would divide the leading coefficient $-2^{21}3^9x_5^5$ of $s_6$
in \eqref{S7lead}. But $r_6$ has degree one in $x_2$, whereas $x_5^5$
is independent of $x_2$. Since $r_6$ is irreducible by
Proposition~\ref{prop:Z}, the polynomials $\mc R_6$ and $s_6$ are coprime.
Proposition~\ref{Prop:B0} gives
$\dim_H(\mc Z(\mc R_6)\cap\mc Z(s_6))\le4$.

The set $\Xi_{E_6}$ is open in $\mc Z(\mc R_6)$ since $s_6$ is continuous.
For $\mb x\in\Xi_{E_6}$, we have $\mc R_6(\mb x)=0$ and $s_6(\mb x)\ne0$.
The leading coefficient of $A^{\mb x}$ in $v$ is $48$.
Thus (3) of Proposition~\ref{prop:appA-subresultant} gives exactly one
common root of $A^{\mb x}$ and $B^{\mb x}$ in $\C$.
The coefficients are real, so the complex conjugate of the common root
is also a common root. Hence the root is real.
\end{proof}

We now determine an explicit formula for the main-analytic function.

\begin{Proposition}\label{thm:E6main}
The polynomial $\mc R_6(\mb x)$ is divisible by $r_6(\mb x)^2$.
Moreover,
$$
\Theta_{E_6}(\mb x):=
\frac{\mc R_6(\mb x)}{r_6(\mb x)^2}
$$
is irreducible over $\C$, and hence over $\Q$.
\end{Proposition}

\begin{proof}
The divisibility follows from Proposition~\ref{prop:appE-square}
in Appendix~\ref{app:exceptional}. Proposition~\ref{prop:appE-irred}
shows that the quotient is irreducible over $\C$, and hence over $\Q$.

\end{proof}

\begin{Corollary}\label{eq:1770}
$\dim_H\bigl(\mc Z(r_6)\cap\mc Z(\Theta_{E_6})\bigr)<5$.
\end{Corollary}

\begin{proof}
The polynomial $r_6$ is independent of $x_0$.
Since $r_6$ and $\Theta_{E_6}$ in $\Q[\mb x]$ are
irreducible,  $r_6$ and
$\Theta_{E_6}$ are coprime. Proposition~\ref{Prop:B0} gives the assertion.
\end{proof}

\begin{Proposition}\label{prop:AE6}
The polynomial $\Theta_{E_6}$ is a main-analytic function for $h(\R^5)$.
\end{Proposition}

To prove the assertion, we prepare the following:

\begin{Lemma}\label{lem:mainE6}
Let $\bm \xi\in \R^6$ satisfy $r_6(\bm \xi)=0$.
If $\bm \xi\in h(\R^5)$,
then $\bm \xi$ belongs to $\mc Z(\Theta_{E_6})$.
\end{Lemma}

\begin{proof}
The real analytic function $r_6\circ h$ is not identically zero, since
\[
h(1,0,0,1,1)=(2,-3,-1,0,1,1),\qquad
r_6(2,-3,-1,0,1,1)=3\ne0.
\]
Hence $\{r_6\circ h\ne0\}$ is dense in $\R^5$.
On this set, Corollary~\ref{eq:1366} and
$\mc R_6=r_6^2\Theta_{E_6}$ give $\Theta_{E_6}\circ h=0$.
Continuity extends this equality to all of $\R^5$.
\end{proof}

\begin{proof}[Proof of Proposition~\ref{prop:AE6}]
Corollary~\ref{eq:1366}, the factorization
$\mc R_6=r_6^2\Theta_{E_6}$, and Lemma~\ref{lem:mainE6} give
$h(\R^5)\subset\mc Z(\Theta_{E_6})$.
Conversely, the second assertion of Lemma~\ref{lem:appE-preimage} gives
\[
\mc Z(\Theta_{E_6})\setminus h(\R^5)
\subset
\bigl(\mc Z(\mc R_6)\cap\mc Z(s_6)\bigr)
\cup\bigl(\mc Z(\Theta_{E_6})\cap\mc Z(r_6)\bigr).
\]
The sets $\mc Z(\mc R_6)\cap\mc Z(s_6)$ and
$\mc Z(\Theta_{E_6})\cap\mc Z(r_6)$ have Hausdorff dimension less than $5$, by
Theorem~\ref{thm:E6} and Corollary~\ref{eq:1770}, respectively.
Since $h(\R^5)$ has dimension $5$ and $\Theta_{E_6}\not\equiv0$,
the assertion follows.
\end{proof}

\begin{proof}[Proof of Theorem~A and (1)--(2) of Theorem~C for type $E_6$]
\leavevmode\par\nobreak
\noindent Theorem~A follows from the preceding results and
Proposition~\ref{thm:E6main}. For (1) of Theorem~C, suppose
that two distinct real source points $(u_i,v_i,\mb x_3)$ have the same
image $\mb x$. If $v_1=v_2=:v$, then $h_2=x_2$ determines $u$ when
$d^{\mb x}(v)\ne0$, while $h_0=x_0$ gives $u_1^3=u_2^3$ when
$d^{\mb x}(v)=0$. Either case contradicts distinctness. Thus $v_1\ne v_2$,
and elimination gives two distinct common roots of $A^{\mb x}$ and
$B^{\mb x}$. Hence $s_6(\mb x)=0$, proving
$h(\mathcal I_h)\subset\mathcal Z(s_6)$.

For (2), Theorem~\ref{thm:E6} and $\mc R_6=r_6^2\Theta_{E_6}$ give
\[
\dim_H\bigl(\mc Z(\Theta_{E_6})\cap\mc Z(s_6)\bigr)
\le\dim_H\bigl(\mc Z(\mc R_6)\cap\mc Z(s_6)\bigr)<5.
\]
Finally, $h^{-1}(\R^6\setminus\mc Z(s_6))$ is open and $h$ is injective
there. If its complement contained a nonempty open set, that set would
contain a regular point of $h$ by Proposition~\ref{prop:1101}. Its image
would have dimension $5$, contradicting (2). This proves density and
completes (1).
\end{proof}

\begin{Remark}\label{rmk:e6}
The polynomial identity $\Theta_{E_6}\circ h=0$ from
Proposition~\ref{prop:AE6} extends to $\C^5$, giving
$h^\C(\C^5)\subset\mc Z_\C(\Theta_{E_6})$.
By Proposition~\ref{Prop:B0},
$\mc L_2:=\mc Z_\C(\Theta_{E_6})\cap\mc Z_\C(r_6)$ has complex dimension
less than $5$. Its complement is therefore dense in the hypersurface
$\mc Z_\C(\Theta_{E_6})$ and is contained in $h^\C(\C^5)$ by
the first assertion of Lemma~\ref{lem:appE-preimage}. 
Properness makes this image closed, so
\begin{equation}\label{eq:E6hc2}
\mc Z_\C(\Theta_{E_6})=h^\C(\C^5).
\end{equation}
Proposition~\ref{thm:E6main} then identifies
$\Theta_{E_6}$, up to a nonzero scalar, with the discriminant polynomial.
For weights $(12,8,9,6,5,2)$ on $(x_0,\ldots,x_5)$, the weighted degrees
of $r_6$, $\mc R_6$, and $\Theta_{E_6}$ are $15$, $102$, and $72$,
respectively.
\end{Remark}

\section{Singularities of type $E_7$}
The standard map
$
h:=h_{E_7}:\R^6\to \R^7
$
of $E_7$-singular points is defined by
$$
h(u,v,\mb x_3):=
\bigl(
h_0(u,v,\mb x_3),
h_1(u,v,\mb x_3),
h_2(u,v,\mb x_3),\mb x_3
\bigr),
$$
where
$\mb x_3 := (x_3,x_4,x_5,x_6)$
and
\begin{align}
h_0&=
2u^3+x_3v^2+2x_4v^3{}+3x_6v^4+uv\,d^{\mb x}(v), \label{eq:1314a}\\
h_1&=-3u^2-v^3-x_5v, \label{eq:1314b}\\
h_2&=-2x_3v-3x_4v^2-4x_6v^3-u\,d^{\mb x}(v), \label{eq:1314c}\\
&d^{\mb x}(v):=3v^2+x_5. \label{eq:d7}
\end{align}
The following three propositions can be proved
in the same way as for type $E_6$.
Set $\K=\R$ or $\C$, and let
$
  \hat h:\K^6\to\K^7
$
denote $h$ when $\K=\R$ and $h^\C$ when $\K=\C$.
The following statement can be proved by
imitating the proof of Proposition~\ref{prop:e6}.

\begin{Proposition}\label{prop:e7}
The image of the standard map $\hat h$ coincides with
the set
\begin{align*}
(\mc W_{E_7}(\K)&:=)\,\Big\{\mb x:=(x_0,\ldots,x_6)\in \K^7\,;\,
\text{there exists $(u,v)\in \K^2$ such that}
\\
& \phantom{aaaaaaaaaaaaaaaaaaaaaa}
\text{$F(u,v,\mb x)=F_u(u,v,\mb x)=F_v(u,v,\mb x)=0$}\Big\},
\end{align*}
where
\begin{align}
F(u,v,\mb x)
&=
u^3+uv^3+x_6v^4+x_5uv
+x_4v^3+x_3v^2+x_2v+x_1u+x_0.
\label{eq:E1047b}
\end{align}
\end{Proposition}

The parameter terms in \eqref{eq:E1047b} 
are given in
\cite[Eq.~(4)]{Vassiliev2024}. Write $F=u^3+uP+Q$, where
\[
 P(v)=v^3+x_5v+x_1,\qquad
 Q(v)=x_6v^4+x_4v^3+x_3v^2+x_2v+x_0.
\]
We note that $F_{uv}=P'=d^{\mb x}$ holds.
A direct Jacobian computation gives the following proposition:

\begin{Proposition}\label{prop:1404}
The map $h$ is a $6$-dimensional real analytic map.
\end{Proposition}

\begin{Proposition}\label{prop:E1073}\label{prop:E7-proper}
Let $\K=\R$ or $\C$, and let
$
  \hat h:\K^6\to\K^7
$
denote $h$ when $\K=\R$ and $h^\C$ when $\K=\C$.
Then $\hat h$ is proper. Moreover, for each $\mb x\in\K^7$,
the inverse image $\hat h^{-1}(\mb x)$ is finite, and
$
  \hat h^{-1}(\mb 0)=\{o\}.
$
\end{Proposition}

\begin{proof}
At $\mb x=\mb0$, \eqref{eq:1314a}--\eqref{eq:1314c} give $u=v=0$.
At any preimage of $\mb x$, \eqref{eq:1314b}--\eqref{eq:1314c} give
\begin{equation}\label{eq:1364a}
u^2=\tfrac13(-v^3-x_5v-x_1),
\end{equation}
and $u d^{\mb x}(v)=-(x_2+2x_3v+3x_4v^2{}+4x_6v^3)$. In particular,
\begin{equation}\label{eq:1364b}
u=\frac{-x_2-2x_3v-3x_4v^2{}-4x_6v^3}{d^{\mb x}(v)}
\qquad\text{if }d^{\mb x}(v)\ne0.
\end{equation}
Squaring the identity for $u d^{\mb x}(v)$ gives $A^{\mb x}(v)=0$
at every preimage, where
\begin{equation}\label{eq:E7-A}
A^{\mb x}(v)
=3(x_2+2x_3v+3x_4v^2{}+4x_6v^3)^2
+d^{\mb x}(v)^2(v^3+x_5v+x_1).
\end{equation}
Since its leading term is $9v^7$, finiteness and properness follow
from the Cauchy-bound argument in Proposition~\ref{prop:E1057},
using \eqref{eq:1364a} to bound $u$.
\end{proof}

Corresponding to \eqref{eq:1314a}--\eqref{eq:1314c}, set
\begin{align}
g^{\mb x}_0(u,v)&:=2u^3{}+3x_6v^4+2x_4v^3+x_3v^2-x_0+uv\,d^{\mb x}(v), \label{eq:1022a7}\\
g^{\mb x}_1(u,v)&:=-3u^2-v^3-x_5v-x_1, \label{eq:1022b7}\\
g^{\mb x}_2(u,v)&:=-4x_6v^3-3x_4v^2-2x_3v-x_2-u\,d^{\mb x}(v). \label{eq:1022c7}
\end{align}
Then $\hat h(u,v,\mb x_3)=\mb x$ is equivalent to
$g^{\mb x}_i(u,v)=0$ $(i=0,1,2)$.
The identity
$2P F_v-3d^{\mb x}F+ud^{\mb x}F_u=2PQ'-3Qd^{\mb x}$
gives $B^{\mb x}(v)=0$ at every preimage, where
\begin{align*}
B^{\mb x}(v)
&=2PQ'-3Qd^{\mb x}\\
&=-x_6v^6-3x_4v^5+(-5x_3+5x_5x_6)v^4
+(-7x_2+3x_4x_5+8x_1x_6)v^3\\
&\phantom{=}+(-9x_0+6x_1x_4+x_3x_5)v^2
+(4x_1x_3-x_2x_5)v-3x_0x_5+2x_1x_2.
\end{align*}
We use the fixed degree bounds $(7,6)$ for $A^{\mb x},B^{\mb x}$.
We define
\begin{equation}\label{R7}
\mc R_7(\mb x):=\op{Res}_v(A^{\mb x},B^{\mb x}),\qquad
s_7(\mb x):=\op{Psc}_v(A^{\mb x},B^{\mb x}).
\end{equation}

The following assertion can be proved
by imitating the proof of Proposition~\ref{prop:1366}.

\begin{Proposition}\label{prop:1322e7}
Fix $\mb x\in\K^7$.
If $\mb x\in\hat h(\K^6)$,
then there exists $v\in\K$ such that
$A^{\mb x}(v)=B^{\mb x}(v)=0$.
Conversely,
if such $v$ exists with
$d^{\mb x}(v)\ne0$,
then $\mb x\in\hat h(\K^6)$.
\end{Proposition}

\begin{Corollary}\label{eq:1366e7}
$h(\R^6)$ is contained in
$\mc Z(\mc R_7)$.
\end{Corollary}

Define
$r_7(\mb x):=(1/3)\op{Res}_v(g_2^{\mb x},d^{\mb x})$
with the fixed degree bounds $(3,2)$.
A direct computation gives
$$
r_7(\mb x)
=
9(x_2-x_4x_5)^2+\frac{4}{3}x_5(3x_3-2x_5x_6)^2.
$$

\begin{Proposition}\label{prop:r7-irred}
The polynomial $r_7(\mb x)$ is irreducible over $\C$.
\end{Proposition}

\begin{proof}
Regard $r_7$ as a quadratic polynomial in $x_2$ over
$R=\C[x_3,x_4,x_5,x_6]$, whose fraction field is denoted by $K$.
Its discriminant is
$$
\Delta_7
=-48x_5(3x_3-2x_5x_6)^2.
$$
The factor $x_5$ occurs to odd order in $\Delta_7$, so $\Delta_7$ is
not a square in $K$.  The polynomial $r_7$ is primitive in $R[x_2]$, and
Lemma~\ref{lem:rj-irred-criterion} implies that $r_7$ is irreducible in
$\C[x_2,x_3,x_4,x_5,x_6]$.
\end{proof}

Considering $\mc R_7$ and $s_7$
as polynomials in $x_0$, we have
\begin{align}\label{T7}
\mc R_7(\mb x)
&=-3^{20}r_7(\mb x)^2x_0^7
+\text{(lower-order terms in $x_0$)}, \\
s_7(\mb x)
&=2^23^{20}J_1(\mb x)J_2(\mb x)x_0^6
+\text{(lower-order terms in $x_0$)}, \label{S7leadB}
\end{align}
where
\begin{equation}\label{T7B}
J_1(\mb x):=3x_3-2x_5x_6,
\quad
J_2(\mb x):=x_2-x_4x_5.
\end{equation}
To verify \eqref{S7leadB}, take the coefficient of $x_0$ from each of
the six rows involving $B^{\mb x}$ in the determinant for $s_7$.
The coefficient of $x_0$ in $B^{\mb x}$ is $-3d^{\mb x}$.
Since $d^{\mb x}=3v^2+x_5$, expansion in the first four columns
gives a factor $9^4$ from the leading coefficient of $A^{\mb x}$.
After removing $(-3)^6$, the remaining determinant is the principal
subresultant coefficient of $A^{\mb x}$ and $d^{\mb x}$ with degree
bounds $7,2$. Subtraction of multiples of $d^{\mb x}$ from
$A^{\mb x}$ reduces the first row to the remainder. Expansion in
that row gives $3^6$ times the coefficient of $v$ in the remainder.
Modulo $d^{\mb x}$ we have
\[
 Q'\equiv J_2+\tfrac{2}{3}J_1v,\qquad
 A^{\mb x}\equiv 3(Q')^2.
\]
The coefficient of $v$ in the remainder is $4J_1J_2$, which gives
\eqref{S7leadB}.
The following is an analogue of Theorem~\ref{thm:E6}.

\begin{Theorem}\label{thm:E7}
The set $\mc Z(\mc R_7)$ has dimension $6$, and
\[
 \dim_H\bigl(\mc Z(\mc R_7)\cap\mc Z(s_7)\bigr)\le5.
\]
The set
\[
 \Xi_{E_7}:=\mc Z(\mc R_7)\setminus\mc Z(s_7)
\]
is open in $\mc Z(\mc R_7)$.
For each $\mb x\in\Xi_{E_7}$, the polynomials $A^{\mb x}$ and
$B^{\mb x}$ have exactly one common root in $\C$, and the root is real.
\end{Theorem}

\begin{proof}
By Corollary~\ref{eq:1366e7} and Proposition~\ref{prop:1404},
$\mc Z(\mc R_7)$ contains the image of $h$, which has dimension $6$.
Since $\mc R_7$ is a nonzero polynomial, $\dim_H\mc Z(\mc R_7)=6$.

As in the proof of Theorem~\ref{thm:E6}, we use the factorization
$\mc R_7=r_7^2\Theta_{E_7}$ over $\Q$. Propositions~\ref{prop:r7-irred}
and~\ref{prop:appE-irred-Q} give the irreducibility of $r_7$ and
$\Theta_{E_7}$, respectively. By \eqref{T7} and \eqref{S7leadB}, the degrees of
$\Theta_{E_7}$ and $s_7$ in $x_0$ are $7$ and $6$. Hence
$\Theta_{E_7}$ does not divide $s_7$.

The degrees of $r_7,J_1,J_2$ in $x_2$ are $2,0,1$, respectively.
The leading coefficient of $s_7$ in \eqref{S7leadB} is a nonzero
constant times $J_1J_2$, so its degree in $x_2$ is one.
Hence $r_7$ cannot divide the leading coefficient of $s_7$.
Since $r_7$ is independent of $x_0$, it cannot divide $s_7$.
Thus $\mc R_7$ and $s_7$ are coprime.
Proposition~\ref{Prop:B0} gives
$\dim_H(\mc Z(\mc R_7)\cap\mc Z(s_7))\le5$.

The set $\Xi_{E_7}$ is open in $\mc Z(\mc R_7)$ since $s_7$ is continuous.
For $\mb x\in\Xi_{E_7}$, we have $\mc R_7(\mb x)=0$ and $s_7(\mb x)\ne0$.
The leading coefficient of $A^{\mb x}$ is $9$, so
(3) of Proposition~\ref{prop:appA-subresultant} applies even if the degree of
$B^{\mb x}$ drops. Hence $A^{\mb x}$ and $B^{\mb x}$ have exactly one
common root in $\C$. Since the coefficients are real, the common root is real.
\end{proof}

\begin{Proposition}\label{thm:E7main}
The polynomial $\mc R_7(\mb x)$ is divisible by $r_7(\mb x)^2$,
and
$$
\Theta_{E_7}(\mb x)
:=
\mc R_7(\mb x)/r_7(\mb x)^2
$$
is irreducible over $\C$, and hence over $\Q$.
\end{Proposition}

\begin{proof}

The divisibility follows from Proposition~\ref{prop:appE-square}
in Appendix~\ref{app:exceptional}. Proposition~\ref{prop:appE-irred}
shows that the quotient is irreducible over $\C$, and hence over $\Q$.
\end{proof}

Imitating the proof of Corollary \ref{eq:1770},
we obtain the following:

\begin{Corollary}\label{eq:1770e7}
The set
$
\mc Z(r_7)\cap \mc Z(\Theta_{E_7})
$
has dimension less than $6$.
\end{Corollary}

\begin{Theorem}\label{Thm:M7}
The polynomial $\Theta_{E_7}(\mb x)$ is a
main-analytic function for $h(\R^6)$.
\end{Theorem}

We prepare the following,
which is an analogue of Lemma~\ref{lem:mainE6}:

\begin{Lemma}\label{lem:mainE7}
If $\bm \xi\in h(\R^6)$ satisfies $r_7(\bm \xi)=0$,
then $\bm \xi$ belongs to $\mc Z(\Theta_{E_7})$.
\end{Lemma}

\begin{proof}
As in the proof of Lemma~\ref{lem:mainE6},
Corollary~\ref{eq:1366e7} and $\mc R_7=r_7^2\Theta_{E_7}$ reduce
the assertion to $r_7\circ h\not\equiv0$.
This follows from
\[
h(1,0,0,0,1,0)=(2,-3,-1,0,0,1,0),\qquad
r_7(2,-3,-1,0,0,1,0)=9\ne0.
\]
\end{proof}

\begin{proof}[Proof of Theorem~\ref{Thm:M7}]
By Corollary~\ref{eq:1366e7} and Lemma~\ref{lem:mainE7},
the factorization $\mc R_7=r_7^2\Theta_{E_7}$ yields
$h(\R^6)\subset\mc Z(\Theta_{E_7})$.
Conversely, the second assertion of Lemma~\ref{lem:appE-preimage}
 implies
\begin{align*}
\mc Z(\Theta_{E_7})\setminus h(\R^6)
&\subset
\mc Z(\Theta_{E_7})\cap\bigl(\mc Z(s_7)\cup\mc Z(r_7)\bigr) \\
&=
(\mc Z(\Theta_{E_7})\cap \mc Z(s_7))\cup
(\mc Z(\Theta_{E_7})\cap \mc Z(r_7)) \\
&\subset
(\mc Z (\mc R_7)\cap \mc Z(s_7))\cup
(\mc Z(\Theta_{E_7})\cap \mc Z(r_7)).
\end{align*}
The two parts of the last subset of $\R^7$
have Hausdorff dimension less than $6$
by Theorem~\ref{thm:E7} and Corollary~\ref{eq:1770e7}, respectively.
Since $h$ is a $6$-dimensional real analytic map,
these inclusions prove the assertion.
\end{proof}

\begin{proof}[Proof of Theorem~A and (1)--(2) of Theorem~C for type $E_7$]
By Theorem~\ref{Thm:M7} and
Proposition~\ref{thm:E7main}, Theorem~A holds.
It remains to prove {\rm(1)} and {\rm(2)} of Theorem~C.

Suppose that two distinct points $(u_i,v_i,\mb x_3)$ $(i=1,2)$
have the same image $\mb x$.
If $v_1=v_2=:v$, then \eqref{eq:1314c} determines $u$ uniquely when
$d^{\mb x}(v)\ne0$, while \eqref{eq:1314a} gives
$u_1^3=u_2^3$ when $d^{\mb x}(v)=0$.
Both contradict distinctness of the preimage, so $v_1\ne v_2$.
Lemma~\ref{LemC1} gives $A^{\mb x}(v_i)=B^{\mb x}(v_i)=0$
for $i=1,2$. Hence $s_7(\mb x)=0$ and
$h(\mathcal I_h)\subset\mathcal Z(s_7)$, proving the first part of (1).

For {\rm(2)}, Theorem~\ref{thm:E7} and
$\mc Z(\Theta_{E_7})\subset\mc Z(\mc R_7)$ give
$
\dim_H\bigl(\mc Z(\Theta_{E_7})\cap\mc Z(s_7)\bigr)<6,
$
proving (2).
Thus $h^{-1}(\R^7\setminus\mathcal Z(s_7))$ is open, injectivity
of $h$ there follows from
$h(\mathcal I_h)\subset\mathcal Z(s_7)$, and density
follows from {\rm(2)} and the fact that $h$ is a $6$-dimensional
real analytic map, proving the second part of (1).
\end{proof}

\begin{Remark}\label{rmk:e7}
Corollary~\ref{eq:1366e7} and Lemma~\ref{lem:mainE7} give
$\Theta_{E_7}\circ h=0$. The polynomial identity extends to $\C^6$.
By \eqref{T7}, the leading coefficient of $\Theta_{E_7}$ in $x_0$
is a nonzero constant. Since $r_7$ is independent of $x_0$, the
two polynomials have no common factor over $\C$. Hence
$\mc Z_\C(\Theta_{E_7})\setminus\mc Z_\C(r_7)$ is dense in
$\mc Z_\C(\Theta_{E_7})$.
The first assertion of Lemma~\ref{lem:appE-preimage}
gives a complex preimage for every point of
$\mc Z_\C(\Theta_{E_7})\setminus\mc Z_\C(r_7)$,
and the properness of $h^\C$ gives
\begin{equation}\label{eq:E7hc2}
\mc Z_\C(\Theta_{E_7})=h^\C(\C^6).
\end{equation}
Since $h^\C$ is an immersion on an open dense subset of $\C^6$,
the hypersurface $h^\C(\C^6)$ is smooth on an open dense subset.
As $\Theta_{E_7}$ is
an irreducible polynomial over~$\C$
(cf. Proposition~\ref{thm:E7main}),
it coincides with the defining polynomial of
$h^\C(\C^6)$ up to a nonzero constant factor.
Thus $\Theta_{E_7}$ gives the discriminant polynomial of type $E_7$
with the normalization in \eqref{T7}.

Assigning weights
$(9,6,7,5,3,4,1)$ to $(x_0,\ldots,x_6)$, we obtain:
\begin{itemize}
\item $r_7(\mb x)$ is a weighted homogeneous polynomial of degree $14$,
\item $\op{Res}_v(A^{\mb x},B^{\mb x})$ is weighted homogeneous of degree $91$,
\item the main-analytic function $\Theta_{E_7}$ is weighted homogeneous of degree $63$.\end{itemize}
\end{Remark}

\section{Singularities of type $E_8$}

The construction for $E_8$ follows the $E_6$ and $E_7$ cases.
We use the corresponding notation $\mc R_8,s_8,r_8,\Theta_{E_8}$
and give the computations specific to this type.

The standard map
$
 h:=h_{E_8}:\R^7\to\R^8
$
of $E_8$-singular points is defined by
$$
h(u,v,\mb x_3):=
\Big(
h_0(u,v,\mb x_3),
h_1(u,v,\mb x_3),
h_2(u,v,\mb x_3),\mb x_3
\Big),
$$
where
$\mb x_3:=(x_3,\ldots,x_7)$ and
\begin{align}
(x_0=)h_0(u,v,\mb x_3)&:=
2 u^3 + 4 v^5 + x_3v^2 + 2x_4 v^3
+ u v\, d^{\mb x}(v), \label{eq:1636a}\\
(x_1=)h_1(u,v,\mb x_3)&:=-3 u^2 - x_5 v - x_6 v^2 - x_7 v^3, \label{eq:1636b}\\
(x_2=)h_2(u,v,\mb x_3)&:=
-5 v^4 - 2x_3v- 3x_4v^2 -u\,d^{\mb x}(v), \label{eq:1636c}\\
d^{\mb x}(v)&:=3x_7 v^2+2 x_6 v+x_5. \label{eq:1687a}
\end{align}
The following two propositions can be proved
in the same way as for type $E_7$.
Set $\K=\R$ or $\C$, and let
$
  \hat h:\K^7\to\K^8
$
denote $h$ when $\K=\R$ and $h^\C$ when $\K=\C$.

\begin{Proposition}\label{prop:e8}\label{prop:E8-basic}
The image of the standard map $\hat h$ coincides with the set
\begin{align*}
(\mc W_{E_8}(\K)&:=)\,\Big\{\mb x:=(x_0,\ldots,x_7)\in \K^8\,;\,
\text{there exists $(u,v)\in \K^2$ such that}
\\
& \phantom{aaaaaaaaaaaaaaaaaaaaaa}
\text{$F(u,v,\mb x)=F_u(u,v,\mb x)=F_v(u,v,\mb x)=0$}\Big\},
\end{align*}
where
$$
F(u,v,\mb x):=
u^3 + v^5 + x_7 u v^3 + x_6 u v^2 + x_5 u v
+ x_4 v^3 + x_3 v^2
+x_2 v+x_1u+x_0.
$$
\end{Proposition}

We note that $F_{uv}=d^{\mb x}$ holds.

\begin{Proposition}\label{prop:1719}\label{prop:e8-dim}
The map $h$ is a $7$-dimensional real analytic map.
\end{Proposition}

\begin{Proposition}\label{prop:E1673}\label{prop:E8-proper}
Let $\K=\R$ or $\C$, and let
$
  \hat h:\K^7\to\K^8
$
denote $h$ when $\K=\R$ and $h^\C$ when $\K=\C$.
Then $\hat h$ is proper. Moreover, for each $\mb x\in\K^8$,
the inverse image $\hat h^{-1}(\mb x)$ is finite, and
$
  \hat h^{-1}(\mb 0)=\{o\}.
$
\end{Proposition}

\begin{proof}
At $\mb x=\mb0$, \eqref{eq:1636b}--\eqref{eq:1636c} give $u=v=0$.
At any preimage, $u d^{\mb x}(v)=-(5v^4+3x_4v^2+2x_3v+x_2)$ and
\begin{equation}\label{eq:E8-u}
u^2=\frac{-x_1-x_5v-x_6v^2-x_7v^3}{3},\qquad
u=\frac{-5v^4-3x_4v^2-2x_3v-x_2}{d^{\mb x}(v)},
\end{equation}
where the second formula requires $d^{\mb x}(v)\ne0$.
Squaring the identity for $u d^{\mb x}(v)$ gives $A^{\mb x}(v)=0$
at every preimage, where
\begin{equation}\label{eq:E8-A}
A^{\mb x}(v):=
3(5v^4+3x_4v^2+2x_3v+x_2)^2
+d^{\mb x}(v)^2(x_1+x_5v+x_6v^2+x_7v^3).
\end{equation}
Its leading term is $75v^8$. The argument of
Proposition~\ref{prop:E1057} therefore gives finite inverse images and
properness, using the first formula in \eqref{eq:E8-u} to bound $u$.
\end{proof}

Corresponding to \eqref{eq:1636a}--\eqref{eq:1636c}, set
\begin{align}
g^{\mb x}_0(u,v)&:=
2 u^3 + 4 v^5 + x_3v^2 + 2x_4 v^3
+ u v\, d^{\mb x}(v)-x_0, \label{eq:1022a8}\\
g^{\mb x}_1(u,v)&:=-3 u^2 - x_5 v - x_6 v^2 - x_7 v^3-x_1, \label{eq:1022b8}\\
g^{\mb x}_2(u,v)&:=
-5 v^4 - 2x_3v- 3x_4v^2 -u\,d^{\mb x}(v)-x_2. \label{eq:1022c8}
\end{align}

Using the two equations in \eqref{eq:E8-u} together with
$h_0(u,v,\mb x_3)=x_0$, we obtain $B^{\mb x}(v)=0$, where
\begin{align}\label{eq:E8-B}
B^{\mb x}(v)&:=
x_7 v^7+4 x_6 v^6+(7 x_5-3 x_4 x_7)v^5 +5 (2 x_1-x_3 x_7)v^4 \\
&\phantom{aaaaa}+(-7 x_2 x_7-2 x_3 x_6+3 x_4 x_5)v^3 \notag\\
&\phantom{aaaaaaaa}+(-9 x_0 x_7+6 x_1 x_4
-4 x_2 x_6+x_3 x_5)v^2  \notag\\
&\phantom{aaaaaaaaaa}+(-6 x_0 x_6+4 x_1 x_3-x_2 x_5)v
-3 x_0 x_5+2 x_1 x_2. \notag
\end{align}
We define
\begin{equation}\label{R8}
\mc R_8(\mb x):=\op{Res}_v(A^{\mb x},B^{\mb x}),\qquad
s_8(\mb x):=\op{Psc}_v(A^{\mb x},B^{\mb x}).
\end{equation}

\begin{Proposition}\label{prop:1322e8}
Fix $\mb x\in\K^8$.  If $\mb x\in\hat h(\K^7)$, then there exists
$v\in\K$ such that
$A^{\mb x}(v)=B^{\mb x}(v)=0$.
Conversely, if such $v$ exists with $d^{\mb x}(v)\ne0$,
then $\mb x\in\hat h(\K^7)$.
\end{Proposition}

\begin{Corollary}\label{eq:1366e8}
$h(\R^7)$
is contained in $\mc Z(\mc R_8)$.
\end{Corollary}

Define
$
r_8(\mb x):=\op{Res}_v(g^{\mb x}_2,d^{\mb x}).
$
A direct computation gives
\begin{align*}
r_8(&\mb x)
= 25\,x_5^4
 \;-\; 90\,x_4 x_7\,x_5^3 \\
&\,\,
 + \bigl(60\,x_4 x_6^2 + 81\,x_4^2 x_7^2 + 180\,x_3 x_6 x_7
+ 90\,x_2 x_7^2\bigr)\,x_5^2 \\
&\,\,\,
 + \bigl(-80\,x_3 x_6^3 - 240\,x_2 x_6^2 x_7 - 108\,x_3 x_4 x_6 x_7^2
          + 108\,x_3^2 x_7^3 - 162\,x_2 x_4 x_7^3\bigr)\,x_5 \\
&\quad\quad
 + x_2\bigl(80\,x_6^4 + 108\,x_4 x_6^2 x_7^2
 - 108\,x_3 x_6 x_7^3 + 81\,x_2 x_7^4\bigr).
\end{align*}

\begin{Proposition}\label{prop:r8-irred}
The polynomial $r_8(\mb x)$ is irreducible over $\C$.
\end{Proposition}

\begin{proof}
Regard $r_8$ as a quadratic polynomial in $x_2$ over
$R=\C[x_3,x_4,x_5,x_6,x_7]$, and let $K$ be the fraction field of $R$.
Its discriminant is
\begin{equation}\label{eq:3611}
-16(3x_5x_7-x_6^2)
\bigl(27x_3x_7^3-27x_4x_6x_7^2+30x_5x_6x_7-20x_6^3\bigr)^2.
\end{equation}
The factor $3x_5x_7-x_6^2$ is irreducible in $R$. Indeed, it is primitive and
linear in $x_5$ over $\C[x_3,x_4,x_6,x_7]$, and its two coefficients
$3x_7$ and $-x_6^2$ are coprime.
By Lemma~\ref{lem:rj-irred-criterion}, the
discriminant given in
\eqref{eq:3611}
 is not a square in $K$.  Since $r_8$
is primitive in $R[x_2]$, Lemma~\ref{lem:rj-irred-criterion}
again shows that $r_8$ is
irreducible in $\C[x_2,x_3,x_4,x_5,x_6,x_7]$.
\end{proof}

Considering $\mc R_8$ and $s_8$
as polynomials in $x_0$, we have
\begin{align}
\mc R_8(\mb x)
  &= c_1
r_8(\mb x)^2 x_0^8
     + \text{(lower-order terms in $x_0$)}, \label{eq:R8lead}\\
s_8(\mb x)
  &= c_2 J_1(\mb x)J_2(\mb x)x_0^7
     + \text{(lower-order terms in $x_0$)}, \label{eq:S8lead}
\end{align}
where $c_1:=3^{15}5^{10}$ and $c_2\in\Z\setminus\{0\}$, and
\begin{align}
J_1(\mb x)
&:= 20\,x_6^3
  + (27 x_4 x_7^2 - 30 x_5 x_7)\,x_6
  - 27 x_3 x_7^3, \label{eq:J1E8}\\
J_2(\mb x) \label{eq:J2E8}
&:= 40\,x_6^4
  + (54 x_4 x_7^2 - 120 x_5 x_7)\,x_6^2
  - 54 x_3 x_7^3\,x_6\\
&\phantom{aaaaaaaaaaaaaaaaaaaa} + \bigl(45 x_5^2 x_7^2
- 81 x_4 x_5 x_7^3 + 81 x_2 x_7^4\bigr).\nonumber
\end{align}
The formulas for $r_8,J_1,J_2$ give
\begin{equation}\label{eq:3650}
\deg_{x_2}r_8=2,\qquad \deg_{x_2}J_1=0,\qquad \deg_{x_2}J_2=1.
\end{equation}
\begin{Theorem}\label{thm:E8}
The set $\mc Z(\mc R_8)$ has dimension $7$, and
\[
 \dim_H\bigl(\mc Z(\mc R_8)\cap\mc Z(s_8)\bigr)\le6.
\]
The set
\[
 \Xi_{E_8}:=\mc Z(\mc R_8)\setminus\mc Z(s_8)
\]
is open in $\mc Z(\mc R_8)$.
For each $\mb x\in\Xi_{E_8}$, the polynomials $A^{\mb x}$ and
$B^{\mb x}$ have exactly one common root in $\C$, and the root is real.
\end{Theorem}

\begin{proof}
Corollary~\ref{eq:1366e8} and Proposition~\ref{prop:1719} give
$\dim_H\mc Z(\mc R_8)\ge7$. Since $\mc R_8$ is a nonzero polynomial,
$\dim_H\mc Z(\mc R_8)=7$.

By Propositions~\ref{prop:appE-square} and~\ref{prop:appE-irred-Q},
$\mc R_8=r_8^2\Theta_{E_8}$ and $\Theta_{E_8}$ is irreducible over $\Q$.
By \eqref{eq:R8lead} and \eqref{eq:S8lead}, the degrees of $\Theta_{E_8}$
and $s_8$ in $x_0$ are $8$ and $7$. Hence $\Theta_{E_8}$ does not divide
$s_8$.

The leading coefficient of $s_8$ in \eqref{eq:S8lead} is a nonzero
constant times $J_1J_2$. By \eqref{eq:3650}, $J_1J_2$ has degree one
in $x_2$ and $r_8$ has degree two. Hence $r_8$ cannot divide the
leading coefficient of $s_8$. Since $r_8$ is independent of $x_0$,
it cannot divide $s_8$. Since $r_8$ is irreducible by
Proposition~\ref{prop:r8-irred}, the factorization
$\mc R_8=r_8^2\Theta_{E_8}$ now shows that $\mc R_8$ and $s_8$ are coprime.
Proposition~\ref{Prop:B0} gives
$\dim_H(\mc Z(\mc R_8)\cap\mc Z(s_8))\le6$.

The set $\Xi_{E_8}$ is open in $\mc Z(\mc R_8)$ since $s_8$ is continuous.
For $\mb x\in\Xi_{E_8}$, we have $\mc R_8(\mb x)=0$ and $s_8(\mb x)\ne0$.
The leading coefficient of $A^{\mb x}$ is $75$.
By (3) of Proposition~\ref{prop:appA-subresultant}, the polynomials
$A^{\mb x}$ and $B^{\mb x}$ have exactly one common root in $\C$.
The common root is real because the coefficients are real.
\end{proof}

\begin{Proposition}\label{thm:E8main}
\emergencystretch=1em
The polynomial $\mc R_8(\mb x)$ is divisible by $r_8(\mb x)^2$.
Moreover, the quotient
$
\Theta_{E_8}(\mb x):=
\mc R_8(\mb x)/r_8(\mb x)^2
$
is irreducible over $\C$, and hence over $\Q$.
\end{Proposition}

\begin{proof}

The divisibility follows from Proposition~\ref{prop:appE-square}
in Appendix~\ref{app:exceptional}. Proposition~\ref{prop:appE-irred}
shows that the quotient is irreducible over $\C$, and hence over $\Q$.

\end{proof}

Imitating the proof of Corollaries \ref{eq:1770}
and \ref{eq:1770e7}, we obtain the following:

\begin{Corollary}\label{eq:1770e8}
The set
$\mc Z(r_8)\cap\mc Z(\Theta_{E_8})$
has dimension less than $7$.
\end{Corollary}

\begin{Theorem}\label{Thm:M8}
The polynomial $\Theta_{E_8}(\mb x)$ is a main-analytic function for
$h(\R^7)$.
\end{Theorem}

We prepare the following, which is an analogue of Lemma~\ref{lem:mainE7}.

\begin{Lemma}\label{lem:mainE8}
Let $\bm\xi\in\R^8$ satisfy $r_8(\bm\xi)=0$.
If $\bm\xi\in h(\R^7)$, then
$\bm\xi\in\mc Z(\Theta_{E_8})$.
\end{Lemma}

\begin{proof}
Apply the proof of Lemma~\ref{lem:mainE7}, using
Corollary~\ref{eq:1366e8}, $\mc R_8=r_8^2\Theta_{E_8}$, and
\[
h(1,0,0,0,1,0,1)=(2,-3,-1,0,0,1,0,1),\qquad
r_8(2,-3,-1,0,0,1,0,1)=16\ne0.
\]
\end{proof}

\begin{proof}[Proof of Theorem~\ref{Thm:M8}]
As in Theorem~\ref{Thm:M7}, Corollary~\ref{eq:1366e8},
$\mc R_8=r_8^2\Theta_{E_8}$, and Lemma~\ref{lem:mainE8} give
$h(\R^7)\subset\mc Z(\Theta_{E_8})$.
By the second assertion of Lemma~\ref{lem:appE-preimage},
\[
\mc Z(\Theta_{E_8})\setminus h(\R^7)\subset\mc E_2\qquad
\Big(\mc E_2:=\mc Z(\Theta_{E_8})\cap
\bigl(\mc Z(s_8)\cup\mc Z(r_8)\bigr)\Big).
\]
Theorem~\ref{thm:E8} and Corollary~\ref{eq:1770e8} bound the dimensions
of the two parts of $\mc E_2$ by a number less than $7$.
Thus $\dim_H\mc E_2<7$, and the assertion follows because $h$ is
a $7$-dimensional real analytic map.
\end{proof}

\begin{proof}[Proof of Theorem~A and (1)--(2) of Theorem~C for type $E_8$]
The main-analyticity of $\Theta_{E_8}$ has been proved in
Theorem~\ref{Thm:M8}, and the last statement of Theorem~A follows from
Proposition~\ref{thm:E8main}.
For parts (1) and (2) of Theorem~C, we follow the proof for type $E_7$.

We only record the $E_8$-specific point needed for {\rm(1)}.  If two
distinct source points have the same image and the same $v$-coordinate,
then \eqref{eq:1636c} determines $u$ when $d^{\mb x}(v)\ne0$, while
\eqref{eq:1636a} gives $2u_1^3=2u_2^3$ when $d^{\mb x}(v)=0$.
Hence their $v$-coordinates must be distinct.  Lemma~\ref{LemC1} then
shows that $A^{\mb x}$ and $B^{\mb x}$ have two distinct common real
roots, so $s_8(\mb x)=0$.  Therefore
$
h(\mathcal I_h)\subset\mathcal Z(s_8).
$
For {\rm(2)}, we have
$$
\mc Z(\Theta_{E_8})\cap\mc Z(s_8)
\subset
\mc Z(\mc R_8)\cap\mc Z(s_8),
$$
and Theorem~\ref{thm:E8} gives
$
\dim_H\bigl(\mc Z(\Theta_{E_8})\cap\mc Z(s_8)\bigr)<7.
$
The openness, injectivity, and density assertion is identical to the
$E_7$ case.
\end{proof}

\begin{Remark}\label{rmk:E8}
The proof of \eqref{eq:E6hc2} applies to type $E_8$.
Theorem~\ref{Thm:M8} gives $\Theta_{E_8}\circ h=0$.
The first assertion of Lemma~\ref{lem:appE-preimage} 
gives a complex preimage for every point of
$\mc Z_\C(\Theta_{E_8})\setminus\mc Z_\C(r_8)$,
and the properness of $h^\C$ gives
\begin{equation}\label{eq:E8hc2}
\mc Z_\C(\Theta_{E_8})=h^\C(\C^7).
\end{equation}
Since $h^\C$ is an immersion on an open dense subset of $\C^7$,
by the same reason as in the cases of $E_6$ and $E_7$,
Proposition~\ref{thm:E8main}
shows that
$\Theta_{E_8}$ coincides with the discriminant polynomial of type $E_8$
up to a nonzero scalar.

Assigning weights $(15,10,12,9,6,7,4,1)$ to $(x_0,\ldots,x_7)$,
we obtain:
\begin{itemize}
\item $r_8(\mb x)$ is a weighted homogeneous polynomial of degree $28$,
\item $\op{Res}_v(A^{\mb x},B^{\mb x})$ is weighted homogeneous of degree $176$,
\item the discriminant $\Theta_{E_8}$ is weighted homogeneous of degree $120$.
\end{itemize}
\end{Remark}

\begin{proof}[Proof of (3) of Theorem~C for types $E_6$, $E_7$, and $E_8$]
Fix $\ell\in\{6,7,8\}$. Let $\mb x\in\R^\ell$ satisfy
$\Theta_{E_\ell}(\mb x)=0$ and $s_\ell(\mb x)\ne0$.
By \eqref{eq:E6hc2}, \eqref{eq:E7hc2}, and \eqref{eq:E8hc2}, the point
$\mb x$ has a complex preimage. Write the preimage as
$(u,v,\mb x_3)$, where $\mb x_3=(x_3,\ldots,x_{\ell-1})$.
The coordinates $x_3,\ldots,x_{\ell-1}$ are real. We prove that $v$ and $u$ are also real:
We first show that $v$ is real.
By Propositions~\ref{prop:e6}, \ref{prop:e7}, and~\ref{prop:e8},
the preimage satisfies
\[
 F(u,v,\mb x)=F_u(u,v,\mb x)=F_v(u,v,\mb x)=0.
\]
Lemma~\ref{LemC1} gives $A^{\mb x}(v)=B^{\mb x}(v)=0$.
Since $s_\ell(\mb x)\ne0$, (3) of Proposition~\ref{prop:appA-subresultant}
shows that $A^{\mb x}$ and $B^{\mb x}$ have only one common root.
Both polynomials have real coefficients, so $\bar v$ is also a common root.
Hence $v=\bar v$.

We next show that $u$ is real. With $v$ fixed, set
$P(U):=F(U,v,\mb x)$. In each of the three types, $P$ is a monic cubic
polynomial with real coefficients. Since $P(u)=P'(u)=0$, the root $u$
has multiplicity at least two. 
If $u$ were nonreal, $\bar u$ would be
a distinct root of multiplicity at least two. The product
$(U-u)^2(U-\bar u)^2$ would then divide $P(U)$. The product has
degree four, but $P$ has degree three. Hence $u$ is real.

We have proved that $h(u,v,\mb x_3)=\mb x$ with
$(u,v,\mb x_3)\in\R^{\ell-1}$, even when $r_\ell(\mb x)=0$.
By (1) of Theorem~C, two distinct real preimages would give $s_\ell(\mb x)=0$.
Hence $\mb x$ has exactly one real preimage.
\end{proof}

\begin{Example}[Rational recovery in type $E_6$]
\label{exa:inverse-E6}
Set $\Theta=\Theta_{E_6}$ and write
$\Theta_{x_j}=\partial\Theta/\partial x_j$.
If $h(u,v,\mb x_3)=\mb x$ and $\Theta_{x_0}(\mb x)\ne0$, the
conormal formula for a generating family \cite[p.~57]{Izumiya2015} gives
\[
 u=\frac{\Theta_{x_1}(\mb x)}{\Theta_{x_0}(\mb x)},\qquad
 v=\frac{\Theta_{x_2}(\mb x)}{\Theta_{x_0}(\mb x)},
\]
since $F_{x_0}=1$, $F_{x_1}=u$, and $F_{x_2}=v$ in
Proposition~\ref{prop:e6}.

\medskip
Consider the real affine line
\[
\mb x(t)=(2160t,0,0,0,72,0),\qquad t\in\R.
\]
The resultant representation of $\Theta$ gives
\[
\Theta(\mb x(t))=C\,t(t^5-1)
\]
for a nonzero real constant $C$. Hence this line meets
$\Delta=\mathcal Z(\Theta)$ at exactly the two real points
$\mb x(0)$ and $\mb x(1)$. Moreover,
\[
2160\,\Theta_{x_0}(\mb x(t))
=\frac{d}{dt}\Theta(\mb x(t))=C(6t^5-1),
\]
so $\Theta_{x_0}$ is nonzero at both points.
Differentiating $\Theta$ in the target coordinates before restricting
to the line gives
\[
\frac{\Theta_{x_1}(\mb x(t))}{\Theta_{x_0}(\mb x(t))}
=\frac{60t^2}{6t^5-1},\qquad
\frac{\Theta_{x_2}(\mb x(t))}{\Theta_{x_0}(\mb x(t))}
=-\frac{30t^4}{6t^5-1}
\quad (6t^5\ne1).
\]
At $t=0$ and $t=1$, the ratios give $(u,v)=(0,0)$ and
$(u,v)=(12,-6)$, respectively. Substitution in
\eqref{eq:1022a}--\eqref{eq:1022c} gives
\[
 h(0,0,0,72,0)=\mb x(0),\qquad
 h(12,-6,0,72,0)=\mb x(1).
\]
Since $h$ preserves the last three coordinates $x_3,x_4$ and $x_5$, 
each point has a
unique real preimage.
\end{Example}

\begin{acknowledgements}
The authors thank Professors
Goulwen Fichou, Toshizumi Fukui,
Atsufumi Honda, Goo Ishikawa, Satoshi Koike and Toru Ohmoto for
valuable comments and for fruitful discussions.
\end{acknowledgements}

\appendix

\section{Resultants and the first subresultant}
\label{app:resultants}

We recall the two subresultant coefficients used in this paper and
specify the matrices needed to compute them. For background, see
Gonz\'alez-Vega~\cite{V} and
Diaz-Toca--Gonz\'alez-Vega~\cite{DGV2004}.
Let $\K=\R$ or $\C$, and write
\[
a(v)=a_0v^n+\cdots+a_n,\qquad
b(v)=b_0v^m+\cdots+b_m,
\]
where $n,m\ge1$ are fixed degree bounds and $n+m\ge3$.
The coefficients may depend on parameters.

The \emph{Sylvester matrix} of $a,b$, denoted by $S(a,b)$, is the $(n+m)\times(n+m)$
matrix whose rows are the coefficient vectors of
\[
v^{m-1}a,v^{m-2}a,\ldots,a,\qquad
v^{n-1}b,v^{n-2}b,\ldots,b,
\]
in the ordered basis $v^{n+m-1},\ldots,v,1$.
If $m=1$ (respectively $n=1$), the first (respectively second)
block consists of $a$ (respectively $b$) alone.
We set $\operatorname{Res}_v(a,b):=\det S(a,b)$.

To compute the first subresultant, set $N:=n+m-2$ and form the
$N\times(N+1)$ matrix $M(a,b)$ whose rows are the coefficient vectors of
\[
v^{m-2}a,\ldots,a,\qquad v^{n-2}b,\ldots,b,
\]
in the ordered basis $v^N,\ldots,v,1$. An empty block is omitted.

Write $M(a,b)=(C_N,\ldots,C_1,C_0)$, with $C_j$ the coefficient column
of $v^j$. Define
\begin{equation}\label{eq:appA-minors}
s_0:=\det(C_N,\ldots,C_2,C_1),\qquad
t_0:=\det(C_N,\ldots,C_2,C_0).
\end{equation}
For $N=1$, these mean $s_0=\det(C_1)$ and $t_0=\det(C_0)$.
Thus $s_0$ is obtained by deleting the constant-term column of $M(a,b)$,
whereas $t_0$ is obtained by deleting its $v$-column.
With this row and column convention, the \emph{first subresultant
polynomial} and its \emph{principal coefficient} are
\begin{equation}\label{eq:appA-Sres}
\operatorname{Sres}_{1,v}(a,b):=s_0v+t_0,\qquad
\operatorname{Psc}_v(a,b):=s_0.
\end{equation}
In particular, $\operatorname{Sres}_{1,v}$ denotes a polynomial of
degree at most one, while $\operatorname{Psc}_v$ denotes its coefficient
of $v$.

For example, if $a(v)=v^3+pv^2+qv+r$ and $b(v)=a'(v)$, then
\[
M(a,b)=
\begin{pmatrix}
1&p&q&r\\
3&2p&q&0\\
0&3&2p&q
\end{pmatrix},\qquad
s_0=6q-2p^2,\quad t_0=9r-pq.
\]
If $a,b$ have a common root and $6q-2p^2\ne0$, that root is
$(pq-9r)/(6q-2p^2)$, as follows from the next proposition.

\begin{Proposition}\label{prop:appA-subresultant}
Assume $(a_0,b_0)\ne(0,0)$. Then:
\begin{enumerate}
\item $\operatorname{Res}_v(a,b)=0$ if and only if $a,b$ have a
common root in $\C$.
\item $\operatorname{Res}_v(a,b)=s_0=0$ if and only if $a,b$ have
a common factor of degree at least two in $\C[v]$.
\item If $\operatorname{Res}_v(a,b)=0$ and $s_0\ne0$, then
$a,b$ have a unique common root $w$, and
\begin{equation}\label{eq:appA-root}
\operatorname{Sres}_{1,v}(a,b)=s_0(v-w),\qquad
w=-\frac{t_0}{s_0}.
\end{equation}
If the coefficients of $a,b$ are real, then $w$ is real.
\end{enumerate}
\end{Proposition}

\begin{proof}
Work over $\C$, and write $g=\gcd(a,b)$ and $d=\deg g$.
Every relation $Ua+Vb=0$ has the form
$(U,V)=((b/g)H,-(a/g)H)$.
Since at least one polynomial has its prescribed degree, the row
relations of $S(a,b)$ correspond to $\deg H<d$, and those of $M(a,b)$ to
$\deg H<d-1$.
Consequently, $S(a,b)$ has nullity $d$, while $M(a,b)$ has row rank
$N-\max\{d-1,0\}$.
This proves (1), and shows that $s_0=t_0=0$ if $d\ge2$.

If $d=1$, let $w$ be the common root.
Then $M(a,b)$ has rank $N$ and annihilates
$(w^N,\ldots,w,1)^{\mathsf T}$.
The last coordinate is $1$, so its first $N$ columns are independent,
giving $s_0\ne0$.
Moreover, $C_0=-\sum_{j=1}^N w^jC_j$. Substituting this into
the determinant for $t_0$ gives $t_0=-ws_0$.
This proves (2)--(3).
\end{proof}

The degree bounds $n,m$ are kept fixed when the parameters are
specialized: zero leading coefficients are retained in the matrices.
Thus $s_0,t_0$ are polynomials in the original coefficients, with no
division required to compute them.
In this paper, $A^{\mb x}$ always has a nonzero constant leading
coefficient. The proposition therefore applies even when
$B^{\mb x}$ loses degree or vanishes identically.

\section{
The dimension of
common zero sets of two polynomials}

\begin{Proposition}\label{Prop:B0}
Let $\mathbb K$ be either $\R$ or $\C$, and let
$a,b\in \Q [x_1,\ldots,x_n]$ be nonzero polynomials, where $n\ge1$.
Suppose that $a$ and $b$ have no nonconstant common factor in
$\Q[x_1,\ldots,x_n]$.
For $p\in\mathbb K[x_1,\ldots,x_n]$, writing
\[
  \mathcal Z_{\mathbb K}(p)
  :=\{\mb x\in\mathbb K^n:p(\mb x)=0\},
\]
we have
\[
  \dim_{\mathbb K}\bigl(
  \mathcal Z_{\mathbb K}(a)\cap\mathcal Z_{\mathbb K}(b)
  \bigr)\le n-2<n-1.
\]
Here $\dim_{\R}$ denotes semialgebraic dimension and $\dim_{\C}$
denotes complex dimension, with $\dim_{\mathbb K}\varnothing=-\infty$.
\end{Proposition}

\section{A property of a system of three polynomial equations}

For $\K=\R$ or $\C$, consider
\begin{align}
G_0(u,v)&:=2u^3+\phi_0(v)+uvd(v), \label{eq:C0}\\
G_1(u,v)&:=-3u^2+\phi_1(v),\qquad
G_2(u,v):=\phi_2(v)-ud(v), \label{eq:C1}
\end{align}
where $d,\phi_0,\phi_1,\phi_2\in\K[v]$. Set
\begin{align}
\alpha(v)&:=3\phi_2(v)^2-d(v)^2\phi_1(v), \label{eq:C4}\\
\beta(v)&:=2\phi_1(v)\phi_2(v)+3\phi_0(v)d(v)
+3v\phi_2(v)d(v). \label{eq:C6}
\end{align}

\begin{Lemma}\label{LemC1}
If $G_j(u_0,v_0)=0$ for $j=0,1,2$ and $(u_0,v_0)\in\K^2$, then
$\alpha(v_0)=\beta(v_0)=0$.
\end{Lemma}

\begin{proof}
This follows, including when $d(v_0)=0$, from the polynomial identities
\[
\alpha=3(\phi_2+ud)G_2-d^2G_1,\qquad
\beta=3dG_0+2udG_1+(2\phi_1+3vd)G_2.
\]
\end{proof}

Conversely, we can show the following:

\begin{Proposition}\label{thm:Cm}
If $d(v_0)\ne 0$ and $\alpha(v_0)=\beta(v_0)=0$
for some $v_0\in \K$,
then there exists $u_0\in \K$ such that
$G_j(u_0,v_0)=0$ for $j=0,1,2$.
\end{Proposition}

\begin{proof}
Set $u_0:=\phi_2(v_0)/d(v_0)$, so $G_2(u_0,v_0)=0$.
The identities in Lemma~\ref{LemC1} then give $G_1(u_0,v_0)=0$
and $G_0(u_0,v_0)=0$, successively.
\end{proof}

\section{Square factors and irreducibility for type E}
\label{app:exceptional}

Consider the case of $E_\ell$ for $\ell=6,7,8$.

\medskip
Set $T:=x_0$.
For the general calculation set $\mb y:=(x_1,\ldots,x_{\ell-1})$.
Propositions~\ref{prop:e6}, \ref{prop:e7}, and~\ref{prop:e8}
give
$$
F=T+G^{\mb y},\qquad G^{\mb y}(u,v):=u^3+u q(v)+g(v),
$$
where
\[
\begin{array}{c|l|l}
 \ell&q(v)&g(v)\\ \hline
 6&x_5v^2+x_4v+x_1&v^4+x_3v^2+x_2v\\
 7&v^3+x_5v+x_1&x_6v^4+{}x_4v^3+x_3v^2+x_2v\\
 8&x_7v^3+x_6v^2+x_5v+x_1&v^5+x_4v^3+x_3v^2+x_2v
\end{array}
\]
Set $p=g'$, $d=q'$, and $f=T+g$, where ${}'=d/dv$.
Then $F=F_u=F_v=0$ can be written as
\begin{equation}\label{eq:3493}
 F=u^3+uq+f=0, \qquad
 F_u=3u^2+q=0,\qquad F_v=ud+p=0.
\end{equation}
At every solution $(u,v)$ of $F=F_u=F_v=0$, the
polynomials $A:=A^{\mb x}$ and $B:=B^{\mb x}$ of Sections~4--6
satisfy $A(v)=B(v)=0$. In the notation above,
\begin{equation}\label{eq:appE-AB}
 A(v)=3p(v)^2+d(v)^2q(v),\qquad B(v)=-3d(v)f(v)+2p(v)q(v).
\end{equation}
Recall that
$$
 \mc R_\ell=\op{Res}_v(A,B).
$$
We remark that the polynomial $d^{\mb x}(v)$ in Sections~4--6 coincides with $d(v)$.

Let $L$ and $M$ be nonzero polynomials in $v$.
Set $n:=\deg_v L$ and $m:=\deg_v M$. Denote the leading coefficient of
$L$ by $a$ and the roots of $L$ by $\alpha_1,\ldots,\alpha_n$,
counting each root with its multiplicity.
The \emph{product formula} for the Sylvester determinant of
Appendix~\ref{app:resultants} is
\begin{equation}\label{eq:appE-root-product}
 \op{Res}_v(L,M)=a^m\prod_{i=1}^n M(\alpha_i).
\end{equation}
We also use the formula for interchanging $L$ and $M$.
\begin{equation}\label{eq:appE-swap}
 \op{Res}_v(L,M)=(-1)^{nm}\op{Res}_v(M,L).
\end{equation}

We compare $\op{Res}_v(p,d)$ with the definitions of
$r_6,r_7,r_8$ in Sections~4--6. Equations~\eqref{eq:1022c1}, \eqref{eq:1022c7}, and
\eqref{eq:1022c8} all take the form
$
  g_2^{\mb x}(u,v)=-p(v)-u\,d(v).
$
Let $\beta$ be a root of $d(v)$. Substituting $v=\beta$ gives
\[
  g_2^{\mb x}(u,\beta)=-p(\beta)-u\,d(\beta)=-p(\beta).
\]
We set $k:=\deg_v d$ and write
$d(v)=\delta\prod_{i=1}^k(v-\beta_i)$, counting roots with
multiplicity. We remark that $p$ and $g_2^{\mb x}$ have
the same degree $j=\deg_v p$ in $v$. By
\eqref{eq:appE-root-product} and the change of order
\eqref{eq:appE-swap}, we obtain
\[
 \op{Res}_v(g_2^{\mb x},d)
=(-1)^{jk}\delta^j\prod_{i=1}^k g_2^{\mb x}(u,\beta_i)
=(-1)^{jk}\delta^j\prod_{i=1}^k\bigl(-p(\beta_i)\bigr)
=(-1)^k\op{Res}_v(p,d).
\]

\medskip
For $E_6$, we have $k=1$, so
$\op{Res}_v(g_2^{\mb x},d)=-\op{Res}_v(p,d)$.
For $E_7$ and $E_8$, we have $k=2$, so
$\op{Res}_v(g_2^{\mb x},d)=\op{Res}_v(p,d)$.
We therefore have
\begin{equation}\label{eq:appE-res-pd}
 \op{Res}_v(p,d)=C_\ell r_\ell,
\end{equation}
where
$C_6:=4$, $C_7:=3$, and $C_8:=1$.

\begin{Proposition}\label{prop:appE-square}
For $\ell=6,7,8$, the polynomial $r_\ell^2$ divides $\mc R_\ell$
in $\Q[\mb y,T]$.
\end{Proposition}

\begin{proof}
The polynomial $A$
has degree $\ell$ in $v$ and has a nonzero constant leading coefficient
$a$, that is
$a=3\cdot 2^4$ if $\ell=6$,
$a=3^2$ if $\ell=7$, and $a=3\cdot 5^2$ if $\ell=8$.
Denote the roots of $A$ by $\alpha_1,\ldots,\alpha_\ell$,
counting each root with its multiplicity.
For each polynomial $P(v)$, set
\[
 \Pi(P):=P(\alpha_1)P(\alpha_2)\cdots P(\alpha_\ell).
\]
The product formula~\eqref{eq:appE-root-product} with $L=A/a$ gives
$\Pi(P)=\op{Res}_v(A/a,P)$. 
 Taking $L=A$ and $M=B$
in \eqref{eq:appE-root-product} gives $\mc R_\ell=a^{\deg_v B}\Pi(B)$.
The degrees of $B$ in $v$ are $5,6,7$ for $E_6,E_7,E_8$,
respectively.

We first compute $\Pi(d)$ for $E_6$. Here $A$ has leading
coefficient $48$ and $d(v)=2x_5v+x_4$. For $x_5\ne0$, set
$\delta:=2x_5$ and $\beta:=-x_4/(2x_5)$, so $d(v)=\delta(v-\beta)$.
Using $A(v)=48\prod_{i=1}^6(v-\alpha_i)$ gives
\begin{equation}\label{eq:4088}
 \Pi(d)=\delta^6\prod_{i=1}^6(\alpha_i-\beta)
       =\frac{\delta^6 A(\beta)}{48}
       =\frac{\delta^6p(\beta)^2}{16},
\end{equation}
where we used the first formula of \eqref{eq:appE-AB}.
By
\eqref{eq:appE-res-pd}
with \eqref{eq:appE-swap}
 and
\eqref{eq:appE-root-product}, it holds that
\begin{equation}\label{eq:4099}
4r_6=\op{Res}_v(p,d)=-\op{Res}_v(d,p)=-\delta^3p(\beta).
\end{equation}
So, by \eqref{eq:4088} and \eqref{eq:4099},
we obtain
\begin{equation}\label{eq:appE-Pi-d-6}
\Pi(d)=r_6^2.
\end{equation}

\medskip
For $E_7$ and $E_8$, the polynomial $d$ has degree two.
Denote its leading coefficient by $\delta$ and its roots by
$\beta_1,\beta_2$, counting a double root twice.
By
\eqref{eq:appE-swap}, we have
\[
 \op{Res}_v(A,d)=\op{Res}_v(d,A),\qquad
 \op{Res}_v(p,d)=\op{Res}_v(d,p).
\]
Apply the product formula~\eqref{eq:appE-root-product}
first to $(d,A)$ and then to $(d,p)$. Since
$A(\beta_1)=3p(\beta_1)^2$ and $A(\beta_2)=3p(\beta_2)^2$, we get
\[
 a^2\Pi(d)=\delta^\ell A(\beta_1)A(\beta_2)
        =9\delta^\ell p(\beta_1)^2p(\beta_2)^2,\qquad
 \op{Res}_v(p,d)=\delta^j p(\beta_1)p(\beta_2).
\]
So, we obtain
\begin{equation}\label{eq:r78}
 \Pi(d)=9a^{-2}\delta^{\ell-2j}\op{Res}_v(p,d)^2=
b_\ell
r_\ell^2
\qquad (b_\ell:=9a^{-2}\delta^{\ell-2j}C_\ell^2).
\end{equation}
For $E_6$, set $b_6:=1$ by \eqref{eq:appE-Pi-d-6}.
For $E_7$, substitute $a=9$, $\delta=3$, $j=3$, and $C_7=3$
in \eqref{eq:r78} to get $b_7=3$.
For $E_8$, substitute $a=75$, $\delta=3x_7$, $j=4$, and $C_8=1$.
So $b_8=9/75^2=1/625$.

\medskip
We now show why $r_\ell^2$ divides $\Pi(B)$.
Set
$
\hat B:=-3df-2qp.
$
Then we have
\[
 B\hat B
=(-3df+2qp)(-3df-2qp)
=9d^2f^2-4q^2p^2.
\]
At a root $\alpha_i$ of $A$, we have $p^2=-d^2q/3$, so
$$
-4q(\alpha_i)^2p(\alpha_i)^2=\tfrac43 d(\alpha_i)^2q(\alpha_i)^3.
$$
Factoring out $d(\alpha_i)^2$ gives
\[
 B(\alpha_i)\hat B(\alpha_i)=d(\alpha_i)^2\bigl(9f(\alpha_i)^2
+\tfrac43q(\alpha_i)^3\bigr).
\]
Taking the product over $i=1,\ldots,\ell$ gives
\[
 \Pi(B)\Pi(\hat B)
  =\Pi(d)^2\,\Pi\bigl(9f^2+\tfrac43q^3\bigr).
\]
Since $\Pi(d)=b_\ell r_\ell^2$ (cf.
\eqref{eq:appE-Pi-d-6} and \eqref{eq:r78}),
the product
$\Pi(B)\Pi(\hat B)$ is divisible by $r_\ell^4$.
It remains to find how many factors of $r_\ell$ can come from
$\Pi(\hat B)$. The coefficient of $T$ in $\hat B$ is $-3d$.
The coefficient of $T^\ell$ in $\Pi(\hat B)$ is therefore
$(-3)^\ell\Pi(d)=(-3)^\ell b_\ell r_\ell^2$. If $r_\ell^3$ divides
$\Pi(\hat B)$, then $r_\ell^3$ would divide every coefficient
of $\Pi(\hat B)$ in $T$, which contradicts
the fact that the coefficient of $T^\ell$
is $(-3)^\ell b_\ell r_\ell^2$.

\medskip
Sections~4--6 prove $r_\ell$ is irreducible over $\C$, so
$r_\ell$ is also irreducible over $\Q$. By unique factorization in
$\Q[\mb y,T]$, the number of copies of $r_\ell$ in a product is the sum
of the numbers in $\Pi(B)$ and $\Pi(\hat B)$. The product
$\Pi(B)\Pi(\hat B)$ contains at least four copies.
The factor $\Pi(\hat B)$ contributes at most two.
Hence $\Pi(B)$ contains at least two copies of $r_\ell$.
Since $\mc R_\ell=a^{\deg_v B}\Pi(B)$, we obtain $r_\ell^2\,\vert\,\mc R_\ell$.
\end{proof}

\medskip
We now compute the leading coefficient in $T$.
The roots of $A$ do not depend on $T$. At each root $\alpha_i$ of $A$,
the coefficient of $T$ in $B(\alpha_i)$ is $-3d(\alpha_i)$.
To obtain the coefficient of $T^\ell$ in $\Pi(B)$, we multiply
the coefficients of $T$ from all $\ell$ factors. Hence
the coefficient of $\mc R_\ell$ with respect to
$T^\ell$ is given by
$c_\ell r_\ell^2$ where
\begin{equation}\label{eq:c_ell}
c_\ell:=a^{\deg_v\!B}(-3)^\ell b_\ell.
\end{equation}
By Proposition~\ref{prop:appE-square}, the quotient
$\Theta_{E_\ell}=\mc R_\ell/r_\ell^2$ is a polynomial.
The coefficient of $T^\ell$ in the quotient is $c_\ell(\ne0)$.
We can therefore write
\begin{equation}\label{eq:appE-H}
 \Theta_{E_\ell}(T,\mb y)=c_\ell{\mc H_\ell(T,\mb y)}
\end{equation}
where $\mc H_\ell(T,\mb y)$ is a monic polynomial in $T$
of degree $\ell$ with coefficients in $\Q[\mb y]$.

\begin{Proposition}\label{prop:appE-irred-Q}
For $\ell=6,7,8$, the polynomial $\Theta_{E_\ell}$ is irreducible over $\Q$.
\end{Proposition}

\begin{proof}

Keep $T=x_0$ as a variable. In each row of \eqref{eq:appE-special},
set the listed coordinates equal to $1$ and all other coordinates of
$\mb y$ equal to $0$. Denote the resulting vector by $\mb y_\ell$.
The last column follows from the formulas for $r_\ell$ in Sections~4--6.
\begin{equation}\label{eq:appE-special}
\begin{array}{c|c|c|c|r}
 \ell&\text{coordinates set to }1&q(v)&g(v)&r_\ell(\mb y_\ell)\\ \hline
 6&x_2,x_5&v^2&v^4+v&-2\\
 7&x_2,x_4&v^3&v^3+v&9\\
 8&x_2,x_4,x_7&v^3&v^5+v^3+v&81
\end{array}
\end{equation}

For $\ell=7$, the specialization has $x_6=0$.
We keep the degree bounds $(7,6)$ for $(A,B)$ and $(3,2)$ for $(p,d)$
as in Appendix~\ref{app:resultants}.

Substituting $a$, $\deg_v B$, and $b_\ell$ in 
\eqref{eq:c_ell}, we have
\begin{equation}\label{eq:appE-leading}
 \mc R_\ell=c_\ell r_\ell^2T^\ell+\text{(lower terms in $T$)},
 \quad
 (c_6,c_7,c_8)=(2^{20}3^{11},-3^{20},3^{15}5^{10}).
\end{equation}
For each $\ell$, we set $N_\ell(T):=\mc R_\ell(T,\mb y_\ell)$. The coefficients of
$N_\ell$ are integers because the resultant is a determinant with integer
polynomial entries.

By \eqref{eq:appE-leading}, the coefficient of $T^\ell$ in
$N_\ell(T)$ is the integer $L_\ell=c_\ell r_\ell(\mb y_\ell)^2$.
Let $\overline{L_\ell}$ denote the residue class of $L_\ell$
modulo $101$. For $\ell=6,7,8$, 
$\overline{L_\ell}=67,64,40$,
respectively.
Let $\overline{N_\ell}(T)$ be the polynomial obtained by reducing
each coefficient of $N_\ell(T)$ modulo $101$.
Since the prime $101$ does not divide $L_\ell$, the element $\overline{L_\ell}$
has an inverse in $\mathbb F_{101}$. Set
\begin{equation}\label{eq:appE-f-mod101}
  f_\ell(T):=\frac{\overline{N_\ell}(T)}{\overline{L_\ell}}
  \quad\text{in }\mathbb F_{101}[T].
\end{equation}
By definition, $f_\ell(T)$ is monic of degree $\ell$.
Then, we have
\[
\begin{aligned}
 f_6(T) &= T^6+36T^5+26T^4+37T^3+8T^2+27T+23,\\
 f_7(T) &= T^7+98T^6+3T^5+44T^4+72T^3+77T^2+54T+57,\\
 f_8(T) &= T^8+88T^7+38T^6+2T^5+32T^4
+37T^3+28T^2+31T+83.
\end{aligned}
\]
One can easily check that
each of $f_\ell(T)$  ($\ell=6,7,8$)
is irreducible over $\mathbb F_{101}$.

\medskip
Suppose $N_\ell$ factors over $\Q$ into two polynomials of positive
degree. By Gauss's lemma we can write $N_\ell=mJK$, where $m$ is
an integer and $J,K\in\mathbb Z[T]$ both have positive degree.
The product of $m$ and the two leading coefficients is $L_\ell$.
Since $101$ does not divide $L_\ell$, both factors keep their degrees
modulo $101$. Reducing $mJK$ would give a nontrivial factorization
of $f_\ell$ after division by the residue of $L_\ell$.
Hence $N_\ell$ is irreducible over $\Q$.
Since $N_\ell(T)=r_\ell(\mb y_\ell)^2\Theta_{E_\ell}(T,\mb y_\ell)$
and $L_\ell=c_\ell r_\ell(\mb y_\ell)^2$, we have
\[
 \mc H_\ell(T,\mb y_\ell)=N_\ell(T)/L_\ell.
\]
In particular, $\mc H_\ell(T,\mb y_\ell)$ is irreducible in $\Q[T]$.

\medskip
If $\mc H_\ell=UV$ in $\Q[\mb y,T]$, the leading
coefficients of $U$ and $V$ in $T$ have product $1$, so both are
nonzero rational constants. Thus two nonconstant factors would keep
their positive degrees in $T$ at $\mb y=\mb y_\ell$, contradicting
the irreducibility of $\mc H_\ell(T,\mb y_\ell)$. Hence $\mc H_\ell$
and its nonzero constant multiple $\Theta_{E_\ell}$ are irreducible
over $\Q$.
\end{proof}

\begin{Lemma}\label{lem:appE-preimage}
Let $\ell\in\{6,7,8\}$, and let $h$ be the standard map of type $E_\ell$.
Suppose that $\mb x\in\C^\ell$ satisfies
$\Theta_{E_\ell}(\mb x)=0$ and $r_\ell(\mb x)\ne0$.
Then $\mb x$ has a preimage under $h^\C$.
If, in addition, $\mb x\in\R^\ell$ and $s_\ell(\mb x)\ne0$,
then $\mb x$ has a real preimage under $h$.
\end{Lemma}

\begin{proof}
By Proposition~\ref{prop:appE-square},
$\op{Res}_v(A,B)=\mc R_\ell(\mb x)
=r_\ell(\mb x)^2\Theta_{E_\ell}(\mb x)=0$.
The leading coefficient of $A$ in $v$ is a nonzero constant.
Thus (1) of Proposition~\ref{prop:appA-subresultant} gives a common
root $w\in\C$ of $A$ and $B$.

We show that $d(w)\ne0$. If $d(w)=0$, then
$0=A(w)=3p(w)^2$ by \eqref{eq:appE-AB}. Hence $p(w)=0$.
A common root of $p$ and $d$ would give $\op{Res}_v(p,d)=0$.
But \eqref{eq:appE-res-pd} gives
$\op{Res}_v(p,d)=C_\ell r_\ell(\mb x)\ne0$.
By the third equation of
\eqref{eq:3493},
we may set $u:=-p(w)/d(w)$ so that $F_v(u,w)=0$.
The converse statements in Propositions~\ref{prop:1366},
\ref{prop:1322e7}, and~\ref{prop:1322e8} give
\[
 h^\C(u,w,x_3,\ldots,x_{\ell-1})=\mb x
\]
for $E_6$, $E_7$, and $E_8$, respectively.

Suppose now that $\mb x\in\R^\ell$ and $s_\ell(\mb x)\ne0$.
By (3) of Proposition~\ref{prop:appA-subresultant},
$w$ is the only common root of $A$ and $B$.
Since $A$ and $B$ have real coefficients, $\bar w$ is also a common root.
Hence $w=\bar w$. The formula $u=-p(w)/d(w)$ now gives $u\in\R$.
Thus $(u,w,x_3,\ldots,x_{\ell-1})$ is a real preimage of $\mb x$.
\end{proof}

\begin{Proposition}\label{prop:appE-irred}
For $\ell=6,7,8$, the polynomial $\Theta_{E_\ell}$ is irreducible over $\C$.
\end{Proposition}

\begin{proof}
We first consider the case $\ell=6$.
By \eqref{eq:E6hc2}, the set
$\mc Z_\C(\Theta_{E_6})=h^\C(\C^5)$ is irreducible, since
$\C^5$ is irreducible and $h^\C$ is a polynomial map.
Consequently, for some $c\in\C\setminus\{0\}$, an irreducible
$P\in\C[x_0,\ldots,x_5]$, and an integer $m\ge1$,
\begin{equation}\label{eq:Pm}
 \Theta_{E_6}=cP^m.
\end{equation}
Since $r_6$ is independent of $x_0$, \eqref{eq:1284a} gives
\[
 \Theta_{E_6}=2^{20}3^{11}x_0^6
 +\text{(lower-degree terms in $x_0$)}.
\]
Thus $\Theta_{E_6}$ is primitive over $R:=\Q[x_1,\ldots,x_5]$.
By Proposition~\ref{prop:appE-irred-Q} and Gauss's lemma,
$\Theta_{E_6}$ is irreducible in $K[x_0]$, where
$K:=\Q(x_1,\ldots,x_5)$.

Since $K$ has characteristic zero, $\Theta_{E_6}$ and
$\partial\Theta_{E_6}/\partial x_0$ are coprime in $K[x_0]$,
and remain so in $\C(x_1,\ldots,x_5)[x_0]$.
If $m\ge2$, then $\deg_{x_0}P>0$ and $P^{m-1}$ divides both
polynomials by \eqref{eq:Pm}, a contradiction.
Hence $m=1$, and $\Theta_{E_6}$ is irreducible over $\C$.

\medskip
We next consider the case $\ell=7$.
By \eqref{eq:E7hc2}, $\mc Z_\C(\Theta_{E_7})=h^\C(\C^6)$ is
irreducible. Thus we can write $\Theta_{E_7}=cP^m$, where
$c\ne0$, $m\ge1$, and $P$ is an irreducible polynomial over $\C$.
The leading term of $\Theta_{E_7}$ in $x_0$ is $-3^{20}x_0^7$
by \eqref{T7}.
The same argument as in the case $\ell=6$ above,
using Proposition~\ref{prop:appE-irred-Q} for $\ell=7$,
gives $m=1$. Thus $\Theta_{E_7}$ is irreducible over $\C$.

\medskip
We finally consider the case $\ell=8$.
Equation~\eqref{eq:E8hc2} gives $\mc Z_\C(\Theta_{E_8})=h^\C(\C^7)$.
Thus we can write $\Theta_{E_8}=cP^m$, where $c\ne0$, $m\ge1$,
and $P$ is an irreducible polynomial over $\C$.
By \eqref{eq:R8lead}, the leading term of $\Theta_{E_8}$ in $x_0$
is $c_1x_0^8$, where $c_1\in\Q\setminus\{0\}$.
The same argument as in the case $\ell=6$ above,
using Proposition~\ref{prop:appE-irred-Q} for $\ell=8$,
gives $m=1$. Thus $\Theta_{E_8}$ is irreducible over $\C$.
\end{proof}

\section{A criterion for global main-analyticity}

\rm
To give a criterion for global main-analyticity as mentioned
in the introduction, we recall the following fact:

\begin{Fact}\label{fact:Coste}
Let
$
    f : \mathbb{R}^m \to \mathbb{R}^n \,\, (m \le n)
$
be a polynomial map.
Assume that $f$ is \emph{generically injective}, that is,
there exists a Zariski open dense subset
$O \subset \mathbb{R}^m$ such that
the restriction $f|_O$ is injective.
Set $S := f(\mathbb{R}^m)$ and let
$\overline{S}^{\,Z} \subset \mathbb{R}^n$ denote the Zariski closure of $S$.
Then the complement $\overline{S}^{\,Z} \setminus S$ has $($algebraic$)$
dimension
strictly less than $m$.
\end{Fact}

\rm
This fact follows from standard results in real algebraic geometry
(see, for instance, Bochnak--Coste--Roy~\cite[\S 2.8 and \S 3.3]{BCR98}).
In our setting, it is obtained as a special case of
Coste~\cite[Corollary~3.19]{Coste2007}, applied to the regular map
$f : \mathbb{R}^m \to \overline{S}^{\,Z}$, where $\mathbb{R}^m$ is regarded
as a nonsingular real algebraic variety.

\begin{Proposition}
\label{thm:Coste-main-analytic-J}
Let
$
    f : \mathbb{R}^m \to \mathbb{R}^n \,\, (m \le n)
$
be a polynomial map.
Assume that there exists a semialgebraic open dense subset
$O \subset \mathbb{R}^m$ such that
\begin{enumerate}
\item[(1)]
$f$ is an immersion on $O$,
\item[(2)]
$f^{-1}(f(x))=\{x\}$ for every $x\in O$.
\end{enumerate}
Then $S := f(\mathbb{R}^m)$ is a global main-analytic set.
Moreover, one may take a main-analytic function $\Theta$ to be a real
polynomial on $\mathbb{R}^n$.
\end{Proposition}

\begin{proof}
Since $O$ is semialgebraic, open, and dense, its complement has dimension
less than $m$, as does its Zariski closure $Z$.
Thus $\R^m\setminus Z$ is a nonempty Zariski open subset of $O$,
on which $f$ is injective by (2). Fact~\ref{fact:Coste} applies.

The immersion condition gives $\dim_H S=m$, and the Zariski closure
$\overline S^{\,Z}$ has the same dimension because $S$ is semialgebraic.
Choose real polynomials $P_1,\ldots,P_N$ defining $\overline S^{\,Z}$
and set $\Theta:=\sum_{i=1}^N P_i^2$.
Then $\mc Z(\Theta)=\overline S^{\,Z}$, and Fact~\ref{fact:Coste} gives
\[
\dim_H\bigl(\mc Z(\Theta)\setminus S\bigr)<m
=\dim_H S=\dim_H\mc Z(\Theta).
\]
Hence $\Theta$ is a polynomial main-analytic function for $S$.
\end{proof}

\begin{Remark}
If $f$ is a real analytic map but not a polynomial,
then the conclusion of Proposition~\ref{thm:Coste-main-analytic-J}
may fail.
For example, the Osgood map
$$
f_O(u,v) := (u,\, uv,\, uve^v) \qquad (u,v\in\R)
$$
is injective on its regular set, but its image cannot be globally
main-analytic.
Indeed, it is a classical fact that any real analytic function
vanishing on $f_O(\R^2)$ near the origin must be identically zero.
Hence no nontrivial real analytic function can have $f_O(\R^2)$ in its
zero set.
\end{Remark}

\end{document}